\documentclass[EJP,preprint]{ejpecp} 
 
\usepackage[utf8]{inputenc} 

\usepackage{mathrsfs,enumitem}   
\usepackage[normalem]{ulem}
\usepackage{comment}
\usepackage{xcolor}

\SHORTTITLE{RWCRE in the sub-ballistic and critical Gaussian regime}

\TITLE{Limiting distributions for RWCRE in the sub-ballistic regime and in the critical Gaussian regime}

\AUTHORS{%
  Conrado~da~Costa\footnote{Durham University, United Kingdom.} \footnote{\EMAIL{conrado.da-costa@durham.ac.uk}. C.dC. was supported through EPSRC grant EP/W00657X/1.}
  \and 
  Jonathon Peterson\footnote{Purdue University, United States of America.} \footnote{\EMAIL{peterson@purdue.edu}. J.P. was supported through a Simons Foundation Collaboration Grant \#635064.}
  \and 
  Yongjia Xie\footnotemark[3]}


\AMSSUBJ{60K37} 
\AMSSUBJSECONDARY{60G50} 

\SUBMITTED{July 14, 2023} 

\VOLUME{0}
\YEAR{2020}
\PAPERNUM{0}
\DOI{10.1214/YY-TN}

\ABSTRACT{Random Walks in Cooling Random Environments (RWCRE) is a model of random walks in dynamic random environments where the environment is frozen between a fixed sequence of times (called the cooling map) where it is resampled. 
  Naturally the limiting distributions for this model depend both on the structure of the cooling sequence and on distribution $\mu$ from which the environments are sampled. Previous results have considered the cases where $\mu$ is such that the corresponding model of random walks in a fixed random environment (RWRE) is either (1) recurrent, (2) has a Gaussian limit with diffusive scaling (the $\kappa > 2$ case), or (3) has positive speed and a stable, non-Gaussian limit (the $\kappa \in (1,2)$ case). 
  
In this paper we examine the limiting distributions in two other transient regimes: the sub-ballistic, non-stable regime (i.e., $\kappa \in (0,1)$), and the Gaussian regime with non-diffusive scaling (i.e., $\kappa = 2$). In the first case we show that the limiting distributions are either Gaussian or a mixture of Gaussian and independent sums of Mittag-Leffler random variables, while in the second case the limiting distributions are always Gaussian but with a scaling that differs from the standard deviation by factor (which can oscillate, but which remains confined to some interval $[\beta,1]$) that depends very delicately on the properties of the cooling map. }

\newcommand{\mc}[1]{{\mathcal #1}}
\newcommand{\bb}[1]{{\mathbb #1}}

\font \mymathbb = bbold10 at 11pt
\newcommand{\N}{\mathbb{N}}
\newcommand{\Mit}{\mathfrak{M}_\kappa}
\newcommand{\cMit}{\widehat{\mathfrak{M}}_\kappa}
\newcommand{\bMit}{\overline{\mathfrak{M}}_\kappa}
\newcommand{\sMit}{\sigma_{\mathfrak{M}}} 
\newcommand{\ind}{\mbox{\mymathbb{1}}}
\newcommand{\fl}[1]{\lfloor{#1}\rfloor} 
\renewcommand{\tilde}{\widetilde}
\DeclareMathOperator{\Var}{Var}
\DeclareMathOperator{\Prob}{\mathbb{P}}
\renewcommand{\P}{\Prob}
\DeclareMathOperator{\E}{\mathbb{E}}
\newcommand{\bPP}{\ensuremath{\overline{\mathbb{P}}}}
\newcommand{\bEE}{\ensuremath{\overline{\mathbb{E}}}}
\def\e{\varepsilon}
\newcommand\ud{\textrm{d}}
\newcommand\veps{\varepsilon}

\begin{document}



\section{Introduction}

Random walks in cooling random environments (RWCRE) are a model of random walks in dynamic random environments introduced by Avena and den Hollander in \cite{AveHol19}. 
This is a model for random motion in an inhomogeneous environment which experiences ``shocks" at certain prescribed times when the entire environment is resampled. 
By adjusting the sequence of times when the environment is resampled, the RWCRE model interpolates between that of random walk in a random environment (RWRE) 
where the environment is ``frozen" and never resampled 
and that of a simple random walk.\footnote{if the environment is resampled on each step, then it is easy to see that the annealed distribution of the RWCRE is the same as that of a simple random walk.} 
The adjective ``cooling" was attached to the model because if the gaps  between successive resamplings of the environment increases without bound then the environment is becoming more and more ``frozen" as time goes on, and in this case we might expect the walk to retain some of the strange asymptotic behaviors of a RWRE. 

Naturally, the behavior of a RWCRE depends on both the distribution $\mu$ that the environments are sampled from and the sequence of times $\tau = \{\tau(n)\}_{n\geq 1}$ (called the cooling map) at which the environment is resampled. Previous results \cite{ACdCdH20,ACP21} have shown that not only can one retain some of the characteristics of a RWRE by choosing a rapidly growing cooling sequence, but that the model of RWCRE can exhibit new limiting distributions (such as tempered stable distributions or sums of independent copies of Kesten-Sinai random variables) which do not occur in either the RWRE model or simple random walks .

The study of limiting distributions of RWCRE has been divided according to the type of the limiting distribution for the RWRE model with distribution $\mu$ on environments. The case where the corresponding RWRE model is recurrent was studied in \cite{AveHol19} and \cite{ACdCdH20}. 
The limiting distributions for transient RWRE are characterized by a parameter $\kappa > 0$ which depends on the distribution $\mu$ on environments \cite{KKS75}. 
The limiting distributions for the cases $\kappa > 2$ (the diffusive, Gaussian regime) and $\kappa \in (1,2)$ (the ballistic, stable regime) were studied in \cite{ACdCdH20} and \cite{ACP21}, respectively. 
In the present paper, we give the limiting distributions for RWCRE in the cases $\kappa \in (0,1)$ (the sub-ballistic, non-stable regime) and $\kappa = 2$ (the Gaussian, super-diffusive regime). The only remaining case ($\kappa=1$) is left for a future work. 
Our main result in the case $\kappa \in (0,1)$ gives examples of new limiting distributions (sums of independent Mittag-Leffler distributions), while our results in the case $\kappa=2$ show that the limiting distribution is always Gaussian but with a non-trivial scaling factor that depends very delicately on the specifics of the cooling map. 

In the remainder of the introduction we will briefly recall the model of one-dimensional RWRE as well as some of the relevant results that we will use. 
Then we will introduce the model of RWCRE and state our main results on the limiting distributions in the cases $\kappa \in (0,1)$ and $\kappa=2$. In the process of proving the limiting distributions for the RWCRE we also obtain some new results for RWRE which may be of independent interest. We state some of these new RWRE results in the introduction as well. 

\subsection{RWRE}
Here we will give a brief overview of the model of one-dimensional RWRE. The interested reader can see the lecture notes of Zeitouni \cite{Z04} or the various references below for more details.

Throughout the paper we use the notation
$\bb{N}_0 := \bb{N} \cup \{0\}$ for the set of non-negative integers.
The classical one-dimensional (static) RWRE model is defined as
follows.  Let $\omega=(\omega(x),x\in\bb{Z}) \in [0,1]^\bb{Z}$ be a
one dimensional environment.  The \emph{random walk in environment}
$\omega$ starting from $z \in \bb{Z}$ is the probability law
$P^\omega_{z}$ on the space of trajectories $\bb{Z}^{\bb{N}_0}$ which
corresponds to the Markov chain $Z = (Z_n)_{n\in\bb{N}_0}$ on $\bb{Z}$
with initial condition $z$ and transition probabilities
\[
P^{\omega}_z(Z_{n+1} = x + e \mid Z_n = x) 
= \left\{
\begin{array}{ll}
\omega(x) &\mbox{ if } e = 1,\\  
1 - \omega(x) &\mbox{ if } e = - 1,
\end{array}
\right. 
\qquad n \in \bb{N}_0.
\]
Let $\mc{G}$ be the sigma algebra on the space of trajectories
$\bb{Z}^{\bb{N}_0}$. By the monotone class theorem, one can verify the
measurability of the map $\omega \mapsto P_z^\omega (G)$ for any
$G \in \mc{G}$ and $z \in \bb{Z}$.  Thus, for any probability measure $\mu$ on the space $[0,1]^\bb{Z}$ of environments we can define the
probability measure $P^\mu_x$ on $\bb{Z}^{\bb{N}_0}$ as the semi-direct product
$P^\mu_x(\cdot):= \mu \ltimes P_z^\omega(\cdot) = \int P_z^\omega(\cdot) \mu (\ud
\omega)$. 
The stochastic process $Z = (Z_n)_{n\in \bb{N}_0}$ is called a RWRE and the distributions $P^\omega_x(\cdot)$ and $P^\mu_x(\cdot)$ are referred to as the \emph{quenched} and \emph{annealed} laws of the RWRE, respectively.

A standard assumption on the distribution $\mu$ on environments, which we will also make here, is that the environments are i.i.d. 
That is,
\begin{equation}\label{iid}
\mu = \alpha^{\bb{Z} },
\end{equation} 
for some probability distribution $\alpha$ on $[0,1]$.
We write $\langle\cdot\rangle$ to denote the expectation w.r.t.\ $\alpha$. 
A crucial quantity to characterize the asymptotic properties of RWRE
is the ratio of the transition probabilities to the left and to the
right at the origin
$\rho_0 = \frac{1 - \omega_0}{\omega_0}$.
For the remainder of the paper, we assume that
\begin{equation}\label{trans}
\langle \log \rho_0\rangle <0,
\end{equation}
which, as shown by Solomon~\cite{S75}, guarantees {\bf right
  transience} for the RWRE; that is, $P^\mu_0( \lim_{n\to\infty} Z_n = \infty ) = 1$.
In addition to \eqref{iid} and \eqref{trans}, we will also assume the conditions on the distribution $\mu$ given by the following definition. 
\begin{definition}[{\bf $\kappa$-regular measures}] \rm{}
  \label{goodenvironment}
  We say that a measure $\mu$ on environments is $\kappa$-\emph{regular} for some $\kappa>0$ if it satisfies \eqref{iid}, \eqref{trans}, the distribution of $\log \rho_0$ is non-lattice, 
\begin{equation}\label{kappadef}
  \langle \rho_0^\kappa\rangle = 1,
\end{equation}  
 and $\langle \rho_0^{\kappa+\e} \rangle < \infty$ for some $\e>0$. 
\end{definition}

\begin{remark}
Since the moment generating function $M(t) = \langle e^{t \log \rho_0} \rangle = \langle \rho_0^t \rangle $ is convex, if \eqref{trans} holds 
then under mild additional assumptions
then there is a (unique) $\kappa>0$ such that 
\eqref{kappadef} holds. 
The additional technical condition $\langle \rho_0^{\kappa+\e} \rangle < \infty$ can be seen as a sort of mild ellipticity condition. 
For some results that we will use, such as the limiting distributions from \cite{KKS75}, a weaker ellipticity condition $\langle \rho_0^\kappa (\log \rho_0)_+ \rangle < \infty$ is sufficient. However, the stronger ellipticity condition $\langle \rho_0^{\kappa+\e} \rangle < \infty$ is needed for the precise large deviation results from \cite{BD18} that will be instrumental in our analysis (see \eqref{prectail-speedcenter}  below). 
\end{remark}

The parameter $\kappa$ given by \eqref{kappadef} characterizes a number of aspects of the asymptotic behavior of the RWRE. 
For instance, 
since the convexity of $t \mapsto \langle \rho_0^t \rangle $ implies that $\kappa > 1$ if and only if $\langle \rho_0 \rangle < 1$, Solomon's LLN for RWRE in \cite{S75} can be written as 
\begin{equation}
  \label{speed-limit}
  \lim_n \frac{Z_n}{n} = v =
  \begin{cases}
    0 & \text{if }\kappa \leq 1 \\
    \frac{1 - \langle \rho_0\rangle}{1 + \langle \rho_0\rangle}&
    \text{if } \kappa > 1,
  \end{cases} \qquad P^\mu_0\text{-a.s.}
\end{equation}
That is, the RWRE is sub-ballistic ($v=0$) if $\kappa \in (0,1]$ and ballistic ($v>0$) if $\kappa > 1$. 
The parameter $\kappa$ also appears in many other results for one-dimensional RWRE, including the characterization of the limiting distributions for transient RWRE \cite{KKS75} and identifying the correct sub-exponential rate of decay for certain large deviation probabilities \cite{DPZ96,GZ98,FGP10,AP16,BD18}. In this paper we will be concerned only with the cases when $\kappa \in (0,1)$ and $\kappa=2$, so next we will recall some of the limiting distribution and large deviation results that are known for these cases.

\subsubsection{The sub-ballistic, non-stable case: \texorpdfstring{$\kappa \in (0,1)$}{kappa in (0,1)}}
The following theorem states the limiting distribution proved in \cite{KKS75} for RWRE in the regime $\kappa \in (0,1)$.
We refer to this as the sub-ballistic, non-stable regime because the walk has limiting speed $v=0$ by \eqref{speed-limit} and the limiting distribution is non-stable -- in contrast to the case $\kappa=1$ where the  walk is sub-ballistic and the limiting distribution is a $1$-stable distribution. 
We note that here and throughout the paper we will use $\Longrightarrow$ to denote convergence in distribution.

\begin{theorem}[\cite{KKS75}]\label{thm:Zn-sub}
 Let $(Z_n)_{n \in \bb{N}_0}$ be a RWRE with distribution $\mu$ on environments that is $\kappa$-regular with $\kappa \in (0,1)$. 
There is a constant $b>0$ such that
under the annealed law
\begin{equation}\label{Znldk1}
 \frac{Z_n}{n^\kappa} \underset{n\to\infty}{\Longrightarrow} \Mit, 
\end{equation}
where $\Mit$ is a non-negative random variable with Laplace transform
\begin{equation}\label{Mkdef}
E[ e^{-\lambda \Mit} ] = \sum_{n=0}^\infty \frac{(-b\lambda)^n}{\Gamma(1+\kappa n)}, \quad \lambda > 0.  
\end{equation}
\end{theorem}

The characterization of the limiting distribution $\Mit$ in \eqref{Mkdef} 
is quite  different from what is given in \cite{KKS75}. Indeed, in \cite{KKS75} the limiting distribution is 
of the form $(S)^{-\kappa}$ where $S$ is a $\kappa$-stable which has Laplace transform $E[e^{-\lambda S}] = e^{-c \lambda^\kappa}$ for some $c>0$. 
However, as can be seen from \cite{Fel71}[Section XIII.8] these two characterizations are equivalent. 
Since the Laplace transform in \eqref{Mkdef} can be written as $\mathcal{E}_\kappa(-b\lambda)$ where 
$\mathcal{E}_\kappa(z) = \sum_{n=0}^\infty \frac{z^n}{\Gamma(1+\kappa n)}$ is the Mittag-Leffler function with parameter $\kappa$, 
we say that $\Mit$ is a Mittag-Leffler random variable. 

\begin{remark}
 There is another family of non-negative random variables which also bear the name ``Mittag-Leffler." These are non-negative random variables $Y_\kappa$ with cumulative distribution functions given by $P(Y_\kappa \leq x) = 1 - \mathcal{E}_\kappa(-x^\kappa)$, for $x\geq 0$, and have Laplace transform $E[e^{-\lambda Y_\kappa}] = \frac{1}{1+\lambda^\kappa}$. 
 To distinguish these two families, the random variables $Y_\kappa$ are said to have Mittag-Leffler distribution \emph{of the first kind}, whereas $\Mit$ are said to have Mittag-Leffler distribution \emph{of the second kind}. 
 Since we will only be concerned with the Mittag-Leffler distributions of the second kind in this paper, we will omit the descriptor ``of the second kind" when referring to $\Mit$ for the remainder of the paper. 
\end{remark}

\subsubsection{The Gaussian, non-diffusive case: \texorpdfstring{$\kappa = 2$}{kappa equals 2}}

The limiting distribution results in \cite{KKS75} show that RWRE with $\kappa$-regular distributions $\mu$ have Gaussian limiting distributions only when the parameter $\kappa \geq 2$. However, the limiting distribution has diffusive $\sqrt{n}$ scaling only when $\kappa>2$ whereas the case $\kappa=2$ has non-diffusive scaling $\sqrt{n\log n}$. 
The following theorem collects this limiting distribution result and some large deviation results that we will use in the remainder of the paper. 

\begin{theorem}\label{thm:Zn-bor}
 Let $(Z_n)_{n \in \bb{N}_0}$ be a RWRE with distribution $\mu$ on environments that is $\kappa$-regular with $\kappa = 2$. Then, the following results hold. 
 \begin{description}
 \item[Limiting distribution \cite{KKS75}] There is a constant $b>0$ such that under the annealed law
 \begin{equation} \label{gaus-static-scaling}
 \frac{Z_n - nv}{b\sqrt{n\log n}} \underset{n\to\infty}{\Longrightarrow} \Phi, 
\end{equation}
where $\Phi$ is a standard Gaussian random variable. 
 \item[Large deviations \cite{CGZ00}] The sequence of random variables $\{Z_n/n\}_{n \geq 1}$ satisfies a large deviation principle with speed $n$ and good, convex rate function $I_\mu(x)$ with 
the property that $I_\mu(x)>0 \iff x \notin [0,v]$.
In particular, for any $\epsilon>0$ there is a constant $C_\epsilon>0$ such that 
\[
P^\mu_0( Z_n - nv > \epsilon n ) \leq e^{-C_\epsilon n}
\quad \text{and} \quad 
P^\mu_0( Z_n < -\epsilon n) \leq e^{-C_\epsilon n},  
\]
for all $n$ large enough. 
 \item[Moderate and large deviation slowdowns \cite{BD18}] There is a constant $K_0>0$ such that 
\begin{equation}\label{prectail-speedcenter}
 \lim_{n\to\infty} \sup_{ \sqrt{n} \log^3 n \leq x \leq nv - \log n } \left| \frac{P^\mu_0(Z_n - nv < -x)}{(nv-x)x^{-2}} - K_0 \right| = 0. 
\end{equation}
 \end{description}
\end{theorem}

\subsection{RWCRE}
In recent years there has been an interest in studying random walks in
\emph{dynamical} random environments; that is, environments which
change over time. A number of results have been able to prove central
limit theorems (Gaussian limiting distributions under diffusive
scaling) by assuming either that the environment has fast enough time
dynamics, see for example \cite{ABF18,BZ06,BMP00}, or by working in a perturbative regime of
certain model parameters where one can prove that the walk moves fast
enough to essentially escape the space time correlations of the
environment, see for example \cite{HdHdSST15,HKT20,HS15}. In contrast, the model of random walks
in cooling random environments first introduced in \cite{AveHol19}
gives a model of a dynamic environment where the time dynamics of the
environment can be made slow enough to retain some of the interesting
effects of the non-Gaussian limiting distributions that one sees with
RWRE.

A random walk in a cooling random environment (RWCRE) is a random walk in a
\emph{space-time} random
environment built by partitioning $\N_0$ into a sequence of intervals,
and assigning independently to each interval an environment sampled
from $\mu$.  Formally, let $(T_k)_{k\in\N}$ be an increment sequence
such that $T_k\in\N$. We will refer to this sequence as {\bf cooling
  increment sequence}.  We denote further by
$\tau(k):=\sum_{i=1}^{k}T_i $ the $k$-th cooling time, i.e. the time
at which a new environment is freshly sampled from $\mu$.  We will
refer to $\tau$ as the {\bf cooling map}. For convenience of notation we will let $\tau(0) = 0$ so that $T_k = \tau(k)-\tau(k-1)$ for all $k\geq 1$.

The RWCRE $(X_n)_{n\in\N_0}$ is defined as follows. 
Let $\bar{\omega} = (\omega^{(k)})_{k\geq 1}= \left( (\omega^{(k)}(x))_{x
  \in\bb{Z}} \right)_{k\geq 1}$ be an i.i.d.\ sequence of environments with
$\omega^{(k)} \sim \mu$. The RWCRE $X$ then starts at $X_0=0$ and evolves on each interval $[\tau(k-1),\tau(k))$ as a random walk in the environment $\omega^{(k)}$. More precisely, given a sequence of environments $\bar\omega$ and the cooling sequence $\tau$ we define the \textbf{quenched} law $P^{\bar\omega,\tau}(\cdot)$ of the RWCRE as that of a (time inhomogeneous) Markov chain with transition probabilities given by 
\[
  \label{cooling_kernel}
P^{\bar{\omega},\tau}(X_{n+1} = x + e \mid X_n = x)=
 \left\{
\begin{array}{ll}
\omega^{(k)}(x) &\mbox{ if } e = 1,\\  
1 - \omega^{(k)}(x) &\mbox{ if } e = - 1,
\end{array}
\right.
\qquad \text{if } \tau(k-1)\leq n < \tau(k). 
\]
The \textbf{annealed} law $\P^{\mu,\tau}(\cdot)$ of the RWRE is then obtained by averaging the quenched law with respect to the measure $\mu^{\N}$ on the sequence of environments $\bar{\omega}$. That is, 
\[
\bb{P}^{\mu,\tau}(\cdot) :=  \mu^{\bb{N}} \ltimes  P_0^{\bar{\omega},\tau}(\cdot) = \int P_0^{\bar{\omega},\tau}(\cdot) \, \ud \mu^{\bb{N}}(\bar{\omega}),
\]
Because we will always be discussing the RWCRE for a fixed distribution $\mu$ and cooling map $\tau$, in a slight abuse of notation we will simply use $\P$ in place of $\P^{\mu,\tau}$ for the annealed law of the RWCRE for the remainder of the paper.

\begin{remark}
Throughout the paper, we will use the following representation of the RWCRE. Let $(Z^{(k)})_{k\geq 1} = \left( (Z^{(k)}_n)_{n\geq 0} \right)_{k\geq 1}$ be a sequence nearest neighbor random walks that are i.i.d.\ with distribution $Z^{(k)} \sim P^\mu_0$; that is each $Z^{(k)}$ is an independent copy of a RWRE with distribution $\mu$ on the environment.
For a fixed $n\geq 1$ 
Let 
\[
\ell_n := \sup\{\ell : \, \tau(\ell) \leq n\} 
\]
be the number of resamplings of the environment by time $n$. 
With this notation, it is easy to see that (under the annealed measure $\P$ for the RWCRE)
\begin{equation}\label{basicdec0}
 X_n \overset{\text{Law}}{=} \sum_{k = 1}^{\ell_
n} Z^{(k)}_{T_k} + Z^{(\ell_n + 1)}_{n-\tau(\ell_n)}, \qquad n \geq 0. 
\end{equation}
(by convention if $\ell_n=0$ then the empty sum on the right side is zero so that $X_n \overset{\text{Law}}{=} Z^{(1)}_{n}$.) 
At times it will be convenient to have some notation to rewrite the right side of \eqref{basicdec0} as a single summation term. To this end, we can write 
\begin{equation}\label{basicdec}
X_n 
\overset{\text{Law}}{=} \sum_k Z^{(k)}_{T_{k,n}}, 
\qquad \text{where} \quad  
T_{k,n} = 
\begin{cases}
 T_k & \text{if } k\leq \ell_n \\
 n-\tau(\ell_n) & \text{if } k = \ell_n + 1 \\
 0 & \text{if } k > \ell_n+1. 
\end{cases}
\end{equation}
\end{remark}

The main results of this paper concern the limiting distributions for the RWCRE. Naturally the limiting distribution depends both on the distribution $\mu$ on environments and the cooling map $\tau$, but it is natural to separate the analysis according to type of the limiting distribution for the RWRE with environment distribution $\mu$. 
In this paper we will be concerned with the cases where $\mu$ is $\kappa$-regular with either $\kappa \in (0,1)$ or $\kappa = 2$. Before stating the results we obtain in these cases, however, we will first review some of the limiting distributions for RWCRE that have already been obtained for other regimes of RWRE.

\noindent\textbf{The Sinai regime: ``$\kappa=0$".\footnote{Since the parameter $\kappa>0$ characterizes the limiting distributions of transient RWRE, in a slight abuse of notation we will refer to the recurrent regime for RWRE (i.e., where $\mu$ is such that $\langle \log \rho_0 \rangle=0$) as having parameter $\kappa=0$.}} 
For recurrent RWRE, Sinai proved the limiting distribution $\frac{Z_n}{(\log n)^2} \Rightarrow V$, where $V$ is a non-Gaussian random variable that can be represented as a functional of a standard Brownian motion \cite{Sinai}.  The limiting distributions for RWCRE for $\mu$ in the Sinai regime were studied first in \cite{AveHol19} for a few special cases of cooling maps and then later in \cite{ACdCdH20} for general cooling maps. The results of \cite{ACdCdH20} showed that all subsequential limits 
of $\frac{X_n-\E[X_n]}{\sqrt{\Var(X_n)}}$
are either Gaussian, sums of independent copies of the random variable $V$, or an independent mixture of Gaussian and sums of independent copies of $V$ (the limiting distribution can depend both on the cooling map $\tau$ and the subsequence $n_j\to\infty$).
Functional limit laws for a few special cases of cooling maps were also obtained in \cite{Yon19}.

\noindent\textbf{The diffusive Gaussian regime: $\kappa>2$.} 
When $\mu$ is $\kappa$-regular with $\kappa>2$ then a CLT-like limiting distribution holds: $\frac{Z_n-nv}{b \sqrt{n}} \Rightarrow \Phi$ for some $b>0$, where $\Phi$ is a standard Gaussian random variable \cite{KKS75}. For $\mu$ in this regime it was shown that for \emph{any} cooling map $\tau$ the limiting distribution for the RWCRE is $\frac{X_n-\E[X_n]}{\sqrt{\Var(X_n)}} \Rightarrow \Phi$. 

\noindent\textbf{The ballistic, stable regime: $\kappa \in (1,2)$.}
When $\mu$ is $\kappa$-regular with $\kappa \in (1,2)$, the limiting distributions for RWRE are of the form $\frac{Z_n-nv}{n^{1/\kappa}} \Rightarrow \mathcal{S}_\kappa$, where $\mathcal{S}_\kappa$ is a $\kappa$-stable that is totally skewed to the left and has mean zero \cite{KKS75}.
Limiting distributions for RWCRE with $\mu$ in this regime were studied in \cite{ACP21} where sufficient conditions were given on the cooling map $\tau$ which lead to limiting distributions for the RWCRE which are (1) Gaussian, (2) $\kappa$-stable of the type $S_\kappa$, (3) generalized tempered $\kappa$-stable, or (4) a mixture of independent random variables of the first three types. 

The main results of the paper concern the limiting distributions of the RWCRE in the cases where the distribution $\mu$ on environments is $\kappa$-regular with either $\kappa \in (0,1)$ or $\kappa = 2$. 
Since both the results and the methods of proof are very different in these two cases we will state our results in each case separately.

\subsubsection{Limiting distributions for the case \texorpdfstring{$\kappa \in (0,1)$}{kappa in (0,1)}}

The decomposition of the RWCRE in \eqref{basicdec} as a sum of increments of independent copies of a RWRE, together with the fact that limiting distributions are known for the RWRE, suggests that one might be able to approximate the distibution of $X_n$ by an appropriate linear combination of independent copies of the limiting distribution of the RWRE (which in the case $\kappa \in (0,1)$ is a Mittag-Leffler random variable $\Mit$). We will refer to this general approach to proving a limiting distribution for the RWCRE as the \textbf{replacement method} approach.
Previous results for RWCRE have shown that the replacement method works sometimes (e.g., for the cases when the RWRE is either recurrent or $\kappa$-regular with $\kappa>2$, \cite{ACdCdH20}) but not always (e.g., when the RWRE is $\kappa$-regular with $\kappa \in (1,2)$, \cite{ACP21}). Our main result for the case $\kappa \in (0,1)$ is that the replacement method does indeed work for this case. 

To prepare for the statement of our main results in this case, note that using \eqref{basicdec} we can rewrite the normalized RWCRE as 
\begin{equation}
 \frac{X_n - \E[X_n]}{\sqrt{\Var(X_n)}} 
 \overset{\text{Law}}{=} \sum_k \frac{Z^{(k)}_{T_{k,n}} - E_0^\mu[Z_{T_{k,n}}]}{\sqrt{\Var(X_n)}} 
 = \sum_{k} \lambda_{\tau,n}(k) \frac{Z^{(k)}_{T_{k,n}} -
   E_0^\mu[Z_{T_{k,n}}]}{\sqrt{\Var(Z_{T_{k,n}})}},
 \label{rm-dec}
\end{equation}
where the coefficients in the last line are given by the vector $\lambda_{\tau,n} = (\lambda_{\tau,n}(k))_{k\geq 1}$ with 
\begin{equation}
    \label{eq:lambda_defi}
    \lambda_{\tau,n}(k)
    =  \sqrt{\frac{\Var(Z_{T_{k,n}})}{\Var(X_n)}}.
\end{equation}
The terms of the vector $\lambda_{\tau,n}$ reflect the relative weight that each 
term in the sum in \eqref{rm-dec}
contributes to the distribution of $X_n$. If the terms of the vector $\lambda_{\tau,n}$ converge to zero uniformly, then it is natural to expect that the limiting distribution of $X_n$ will be Gaussian. On the other hand, if some terms of $\lambda_{\tau,n}$ remain bounded away from zero then we expect the Mittag-Leffler random variables $\Mit$ to appear in the limiting distribution. To make this precise, and state our main results in the case $\kappa \in (0,1)$ we need to first introduce some notation. 

Let $\ell^{2} = \{ \mathbf{x} \in \mathbb{R}^\N : \sum_{k\geq 1} x(k)^2 < \infty \}$ be the collection of square summable sequences, and note that $\lambda_{\tau,n} \in \ell^{2}$ since $\sum_{k\geq 1} \lambda_{\tau,n}(k)^2 = 1$. For any non-negative sequence $\mathbf{x} \in \ell^{2}$ there exists a unique non-increasing sequence $\mathbf{x}^\downarrow \in \ell^{2}$ that is a re-ordering of the terms of $\mathbf{x}$.\footnote{That is $\mathbf{x}^\downarrow = (x^\downarrow(k))_{k\geq 1}$ is the unique element of $\ell^{2}$ such that 
$x^\downarrow(\cdot) = x(\pi(\cdot))$ 
for some bijection $\pi: \N \to \N$ and such that $x^\downarrow(k)\geq x^\downarrow(k+1)$ for all $k\geq 1$.}
Finally, for any random variable $Z$ with finite variance let $\widehat{Z} = \frac{Z-E[Z]}{\sqrt{\Var(Z)}}$ denote the normalized version of $Z$ and for any $\mathbf{x} \in \ell^{2}$ let 
\[
 \left( \widehat{Z} \right)^{\otimes \mathbf{x}} = \sum_{k\geq 1} x(k) \widehat{Z}_k, 
\]
where $\widehat{Z}_1,\widehat{Z}_2,\ldots$ are i.i.d.\ copies of the random variable $\widehat{Z}$. 

Having introduced the necessary notation, we are now ready to state our main result in the case $\kappa \in (0,1)$ which says that subsequential limits of RWCRE when $\kappa \in (0,1)$ are sums of independent Mittag-Leffler random variables, Gaussian random variables, or a mixture of the two. 

\begin{theorem}\label{thm:mit_mixtures}
 Let $X_n$ be a RWCRE with $\kappa$-regular distribution $\mu$ with $\kappa \in (0,1)$ and cooling map $\tau$. 
 Assume that $n_j \to \infty$ is a subsequence such that 
 $\lim_{j\to\infty} \lambda_{\tau,n_j}^\downarrow(k) = \lambda_*(k)$ for all $k\geq 1$ for some $\lambda_* \in \ell^{2}$. 
 Then, 
\begin{equation}
 \label{Lawsub}
 \frac{X_{n_j}-\E[X_{n_j}]}{\sqrt{\Var(X_{n_j})}}  
 \underset{j\to\infty}{\Longrightarrow}
\left( \cMit \right)^{\otimes \lambda_*}
+ a(\lambda_*)\Phi,
\end{equation}
where $a(\lambda_*) := \left( 1 - \sum_k \lambda^2_*(k)\right)^{1/2} \in [0,1]$  and
  $\Phi$ is a standard Gaussian random variable independent from
  $\left( \cMit \right)^{\otimes \lambda_*}$.  Moreover, the
  convergence in law also holds in $L^p$ for all $p>0$.
\end{theorem}

\begin{remark}
 While Theorem \ref{thm:mit_mixtures} shows that subsequential limiting distributions can be mixtures of Mittag-Leffler and Gaussian distributions, only the Gaussian distributions can arise as limits of the full sequence.
  To see this, first note that if $ \mathbf{x}, \mathbf{y} \in \ell^2$ and
  $\mathbf{x}^\downarrow \neq \mathbf{y}^\downarrow$ then the random
  variables
  $\left( \cMit \right)^{\otimes \mathbf{x}} + a(\mathbf{x})\Phi $ and
  $\left( \cMit \right)^{\otimes \mathbf{y}} + a(\mathbf{y})\Phi $ are
  distinct. This can be seen by examining the Laplace transform of the
  above random variables see also the proof of Lemma 1 in \cite[\S
  3.2]{ACdCdH20}. Now, since from every sequence
  $(\lambda_n, n \in \bb{N})$ in $\ell^2$ one can extract a
  subsequence $n_j$ for which
  $(\lambda_{n_j}^\downarrow(k), n \in \bb{N})$ converges for all
  $k \in \bb{N}$, it follows from \eqref{Lawsub} that the limit laws
  of the sequence
  $\bigg(\frac{X_{n}-\E[X_{n}]}{\sqrt{\Var(X_{n})}}, n \in \bb{N}
  \bigg)$ are in correspondence with the limit points of the sequence
  $(\lambda_{\tau,n}^\downarrow, n \in \bb{N})$ given by
  $\{\lambda_* \in \ell^2: \lim_n \lambda_{\tau,n}^\downarrow(k) =
  \lambda_*(k) \text{ for all } k \in \bb{N}\}$.
  Moreover, the full sequence
  $ \bigg(\frac{X_{n}-\E[X_{n}]}{\sqrt{\Var(X_{n})}}, n \in \bb{N}
  \bigg)$ converges if and only if the sequence $\lambda_{\tau,n}^\downarrow$
  admits a unique limit $\lambda_*$. 
  This happens if and only if
  $\lambda_*(k) = 0$ for all $k \in \bb{N}$ (or equivalently
  $\lim_k \lambda_{\tau,\tau(k)}(k) = 0$) since the vector $\lambda_{\tau,n}$ records the proportion of the variance coming from each cooling interval up to time $n$.
\end{remark}

\begin{remark}
  Theorem \ref{thm:mit_mixtures} shows that any subsequential limit of
  the RWCRE must be a random variable of the form in the right side of
  \eqref{Lawsub} for some $\lambda_*$. Indeed, using a diagonalization
  argument it is easy to see that for any subsequence there is always
  a further subsequence so that $\lambda_{\tau,n_j}^\downarrow(k)$
  converges for all $k$ (and the limiting vector $\lambda_*$ must be
  in $\ell^2$ by Fatou's Lemma).  Moreover, we give explicit examples
  in Section \ref{sec:ex} which show that that the limit can be a pure
  Gaussian ($a(\lambda^*) = 1$), a pure mixture of centered
  Mittag-Leffler random variables ($a(\lambda_*) = 0$) or a mixture of
  the two ($a(\lambda_*) \in (0,1)$).
Finally, it is a natural question as to whether or not 
there are restrictions on what mixtures of Mittag-Leffler and Gaussian random variables one can obtain as subsequential limits of RWCRE when $\kappa \in (0,1)$. 
In Example \ref{ex:arbMLGmix} answer this question by giving an algorithm which shows that for 
\emph{any} non-negative $\lambda_* \in \ell^2$ with $\sum_k \lambda_*(k)^2 \leq 1$ we can construct a cooling map $\tau$ and a subsequence $n_j$ such that \eqref{Lawsub} holds. 
\end{remark}

Theorem \ref{thm:mit_mixtures} as stated is quite general. However, to check the convergence of $\lambda^{\downarrow}_{\tau,n}$ along some subsequence one needs control on the variance of the corresponding RWRE. 
The following theorem, which is the key to the proof of
Theorem \ref{thm:mit_mixtures},
also provides the necessary asymptotics on the variance to be able to identify the (possibly subsequential) limiting distributions for specific choices of cooling maps $\tau$.

\begin{theorem}[\textbf{RWRE $L^p$ Convergence $\kappa\in(0,1)$}]\label{thm:lp_cov}
Let $Z=(Z_n)_{n\geq 0}$ be a RWRE with distribution $\mu$ on environments that is $\kappa$-regular with $\kappa \in (0,1)$. 
Then the convergence in distribution in \eqref{Znldk1} also holds in $L^p$ for all $p>0$. 
In particular $E_0^\mu[Z_n] \sim \mu_{\mathfrak{M}} n^\kappa$ and $\Var(Z_n) \sim \sMit^2 n^{2\kappa}$ as $n\to\infty$, where
\[
\mu_{\mathfrak{M}} := E[\Mit] = \frac{b}{\Gamma(1+\kappa)} 
\quad\text{and}\quad 
\sMit^2 := \Var(\Mit) = b^2 \left( \frac{2}{\Gamma(1+2\kappa)} - \frac{1}{\Gamma(1+\kappa)^2} \right). 
\]
\end{theorem}


In the case where $T_k \to \infty$ (that is, the gaps between resampling times of the environment diverge),
Theorem \ref{thm:lp_cov} implies that 
\begin{equation}\label{Vn-stdev}
 \lim_{n\to\infty} \frac{\Var(X_n)}{\sum_k (T_{k,n})^{2\kappa}} = \sMit^2. 
\end{equation}
In fact, the condition \eqref{Vn-stdev} (which can hold even if $T_k \not\to \infty$) is sufficient to give the following
more explicit way to check the conditions for the subsequential limiting distributions in \eqref{Lawsub}. 
\begin{corollary}\label{Cor:lambda_explicit}
Let $X_n$ be a RWCRE satisfying the assumptions of Theorem \ref{thm:mit_mixtures}. If in addition the cooling map is such that 
\eqref{Vn-stdev} holds,
then the conclusion of Theorem \ref{thm:mit_mixtures} holds true if $\lambda_{\tau,n}$ is replaced by the vector $\tilde{\lambda}_{\tau,n}$ defined by 
\[
    \tilde{\lambda}_{\tau,n}(k) = 
      \frac{(T_{k,n})^\kappa}{V_n},
    \quad \text{where} \quad
    V_n = \sqrt{\sum_j (T_{j,n})^{2\kappa} },
\]
where $T_{k,n}$ is defined as in \eqref{basicdec}.
\end{corollary}
\begin{remark}
  The condition \eqref{Vn-stdev} is sufficient but not necessary in order to replace $\lambda_{\tau,n}$ with $\tilde{\lambda}_{\tau,n}$, as can be seen by considering 
  the cooling map with $T_k \equiv 1$ in which case $\lambda_{\tau,n}^\downarrow$ and $\tilde{\lambda}_{\tau,n}^{\downarrow}$ both converge pointwise to $\mathbf{0} \in \ell^2$.
  \end{remark}

\begin{remark}\label{rem:replacemean}
  In light of the asymptotics of $\E[Z_n]$ in Theorem \ref{thm:lp_cov} we may consider the need for
  the centering term in  Theorem
  \ref{thm:mit_mixtures}. For any random variable $Z$ with finite
  variance, let $\bar{Z}= \frac{Z}{\sqrt{\Var(Z)}}$ denote the non-centered
  normalized version of $Z$ and for $\mathbf{x} \in \ell^2$ let
  $\left(\bar{Z} \right)^{\otimes \mathbf{x}} = \sum_{k \geq 1}
  x(k) \bar{Z}_k$ where $\bar{Z}_1,\bar{Z}_2,\ldots$ are i.i.d. copies
  of the random variable $\bar{Z}$. 
  Since
  \begin{align*}
    \frac{X_{n_j}}{\sqrt{\Var(X_{n_j})}} =
    \frac{X_{n_j}-\E[X_{n_j}]}{\sqrt{\Var(X_{n_j})}} + \frac{\E[X_{n_j}]}{\sqrt{\Var(X_{n_j})}}
  \end{align*}
  it follows from  \eqref{Lawsub} that if $\lim_{j\to\infty}
  \lambda_{\tau,n_j}^\downarrow(k) = \lambda_*(k)$ for all $k\geq 1$
  then $
  \frac{X_{n_j}}{\sqrt{\Var(X_{n_j})}}$ converges to $\left(\bMit
  \right)^{\otimes \lambda_*} +  a(\lambda_*)\Phi$ 
  if $\sum_k \lambda_*(k)<\infty$ and 
  \[
    \lim_{j\to\infty} \frac{\E[X_{n_j}]}{\sqrt{\Var(X_{n_j})}} =  \bb{E}\bigg[\left(\bMit
  \right)^{\otimes \lambda_*}\bigg]
  = \frac{\mu_{\mathfrak{M}}}{\sMit}\left( \sum_k \lambda_*(k)\right).
  \]
\end{remark}

\subsubsection{Limiting distributions for the case \texorpdfstring{$\kappa =2$}{kappa equals 2}}


The limiting distribution result for the RWCRE in the case $\kappa=2$ (Theorem \ref{thm:RWCRElim-k2}) inherits some of the properties of both the case $\kappa>2$ and $\kappa \in (1,2)$. 
Like the $\kappa>2$ case, the limit will be Gaussian for \emph{any} cooling map.
On the other hand, like the case $\kappa \in (1,2)$ one cannot use the replacement method to prove limiting distributions and determining the proper scaling for the limiting distribution is a major difficulty.

In order to use the replacement method to prove a limiting distribution for the RWCRE, one needs to improve the limiting distribution for the RWRE to convergence in $L^2$ (see the discussion on the replacement method in Section \ref{ssec:thm:mit_mix}). 
The first main result in this section gives new asymptotics on the variance of the RWRE in the case $\kappa=2$, and as a consequence shows that one does not have $L^2$ convergence in this case.


\begin{theorem}\label{thm:EZn2lim}
Let $Z=(Z_n)_{n\geq 1}$ be a RWRE with distribution $\mu$ on environments that is $\kappa$-regular with $\kappa =2$. 
Then, 
\begin{equation}\label{EZn2lim}
\lim_{n\to\infty} E_0^\mu \left[ \left( \frac{Z_n - nv}{\sqrt{n\log n}} \right)^2 \right] = b^2 + K_0 v, 
\end{equation}
where $v>0$ is the limiting speed as in \eqref{speed-limit} and the constants $b$ and $K_0$ are as in \eqref{gaus-static-scaling} and \eqref{prectail-speedcenter}, respectively. 
\end{theorem}

While \eqref{EZn2lim} shows that the limiting distribution in \eqref{gaus-static-scaling} cannot be improved to $L^2$-convergence, 
it also implies that $\left| \frac{Z_n-nv}{\sqrt{n\log n}}\right|^p$ is uniformly integrable for any $p \in (0,2)$, and thus the convergence in distribution in \eqref{gaus-static-scaling} can be improved to $L^p$ convergence for all $p \in (0,2)$. 
In particular, this implies the following asymptotics for the mean and variance of the RWRE. 

\begin{corollary}\label{k2-meanvar}
 Under the same assumptions as Theorem \ref{thm:EZn2lim}, we have 
\begin{equation}\label{eq:2_mean_var_scaling}
 \lim_{n\to\infty} \frac{E_0^\mu[Z_n]-nv}{\sqrt{n\log n}} = 0, 
 \quad\text{and}\quad 
 \lim_{n\to\infty} \frac{\Var(Z_n)}{n\log n} = b^2 + K_0 v. 
 \end{equation}
\end{corollary}
The asymptotics of the mean and variance in Corollary \ref{k2-meanvar} imply that we can restate the limiting distribution in \eqref{gaus-static-scaling} as 
\begin{equation}\label{Zn-limdist-st}
 \frac{Z_n - E_0^\mu[Z_n]}{\beta \sqrt{\Var(Z_n)}} \underset{n\to\infty}{\Longrightarrow} \Phi, 
 \quad \text{where } \beta = \frac{b}{\sqrt{b^2 + K_0 v}} < 1.  
\end{equation}
Of course the interesting part of the limiting distribution as stated in \eqref{Zn-limdist-st} is that the constant $\beta<1$. 
The fact that the RWRE must be scaled my a non-trivial multiple of the standard deviation to get a standard Gaussian limit is then reflected in our main result for the limiting distributions for RWCRE in the case $\kappa=2$ where the appropriate multiplicative constant depends very delicately on the cooling map $\tau$ and the distribution $\mu$. 

\begin{theorem}\label{thm:RWCRElim-k2}
Let $X_n$ be a RWCRE with $2$-regular distribution $\mu$ and cooling map $\tau$. 
There exists a sequence of numbers $\beta_n = \beta_n(\mu,\tau)  \in [\beta,1]$ for $n\geq 1$ such that 
\[
\frac{X_n - \E[X_n]}{\beta_n \sqrt{\Var(X_n)}} \Rightarrow \Phi,
\] 
where $\Phi$ is a standard normal random variable. 
\end{theorem}
\begin{remark}
 The formula for the scaling constants $\beta_n$ in terms of the cooling map $\tau$ and the distribution $\mu$ is given explicitly in \eqref{tbnform} below and involves certain truncated variance terms for the RWRE of the form $\Var((Z_n - E_0^\mu[Z_n])\ind_{|Z_n - E_0^\mu[Z_n]|\leq x })$. As part of the proof of Theorem \ref{thm:RWCRElim-k2} we will also give precise asymptotics for such truncated variance terms, and thus one can compute the scaling constants $\beta_n$ for certain specific choices of cooling maps.  In particular, we will give examples in Section \ref{sec:ex} which show that the constants $\beta_n$ can fill the entire range from $\beta$ to $1$ and that the sequence $\beta_n$ can also oscillate with $n$. 
\end{remark}

\begin{remark}
 One can sometimes use the asymptotics of the mean and variance of the corresponding RWRE in Corollary \ref{k2-meanvar} to replace the scaling and/or centering terms in Theorem \ref{thm:RWCRElim-k2} with more explicit expressions. 
 For instance, if $\lim_{k\to\infty} T_k = \infty$ one can replace $\sqrt{\Var(X_n)}$ with $\sqrt{\sum_{k} T_{k,n} \log(T_{k,n})}$ and if in addition one has 
  \begin{equation}\label{er2}
  \sup_{n} \frac{\sum_{k= 1}^{\ell_n+1} \sqrt{T_{k,n} \log(T_{k,n})} }{\sqrt{\sum_{k= 1}^{\ell_n+1} T_{k,n} \log(T_{k,n})}} < \infty
 \end{equation}
then one can also replace the centering term $\E[X_n]$ with $nv$. 
In fact, it is not hard to see that \eqref{er2} is equivalent to the slightly easier to check 
\begin{equation}\label{errorratio}
  \sup_n \frac{\sum_{k=1}^n \sqrt{T_k \log(T_k)}}{\sqrt{\sum_{k=1}^n T_k \log(T_k)} } < \infty, 
\end{equation}
which differs from \eqref{er2} only in that one doesn't have to consider the partial cooling interval $T_{\ell_n+1,n} = n - \tau(\ell_n)$. 
\end{remark}

The proof of Theorem \ref{thm:RWCRElim-k2} is the most difficult and innovative part of the paper. There are two natural ways to try to prove Gaussian limits for the RWCRE, but neither works for all cooling maps. 

\noindent\textbf{Approach 1:} The first approach is to try to apply the Lindberg-Feller CLT using the representation in \eqref{basicdec} as a sum of independent random variables. This approach will work only if the cooling map grows slowly enough so that the triangular array is \emph{uniformly asymptotically negligible}, i.e., 
\[
    \lim_{n \to \infty}\max_{k \leq \ell_n + 1} \bb{P}\big(Z^{(k)}_{T_{k,n}}>\veps \sqrt{\Var(X_n)}\big)=0,
\]
see also \cite[\textsection 7.3, p. 313]{Ash00}. For instance, this approach will work for cooling maps with $T_k \sim A k^{\alpha}$ for some $A,\alpha>0$ (i.e., polynomial cooling). Even in this case, applying the Lindeberg-Feller CLT is not always straightforward as one sometimes needs to apply a truncation step first before applying the CLT for triangular arrays (e.g., polynomial cooling with $\alpha>1$). 

\noindent \textbf{Approach 2:} If the cooling map grows sufficiently fast, then the distribution of $X_n$ is essentially controlled by the last few terms of the sum in \eqref{basicdec}. 
For instance, this approach will work for cooling maps with $T_k \sim A e^{ck}$ for some $A,c>0$ (i.e., exponential cooling). 
The idea with this approach is that one first fixes $m$, and then notes that the sum of the largest $m$ terms in \eqref{basicdec} converges in distribution (after appropriate centering and scaling) to a Gaussian. Then one argues that for cooling maps growing fast enough the distribution of the sum of the largest $m$ terms in \eqref{basicdec} is not very different from the distribution of $X_n$ if $m$ is large enough. 

Since the first approach above only works for cooling maps growing sufficiently slowly and the second approach only works for cooling maps growing sufficiently fast, it is not obvious how to prove Gaussian limits for cooling maps which are more irregular. In our proof of Theorem \ref{thm:RWCRElim-k2} we show how the two approaches can be combined together to cover general cooling maps. One splits the sum in \eqref{basicdec} into the terms where $T_k$ is ``large" or ``small", respectively (whether $T_k$ is classified as ``large" or ``small" depends on its relative size among all the other terms in the sum) and then simultaneously applies approach 1 to the ``small" terms and approach 2 to the ``large" terms. 
There are two main difficulties to implementing this approach for general cooling maps. 
The first difficulty is finding the appropriate way to divide the ``small" and ``large" terms so that both approaches can be applied simultaneously. 
The second difficulty is that when applying approach 1 to the ``small" terms one still needs to truncate the terms before applying the Lindeberg-Feller CLT, and it is a very delicate matter to choose a truncation that works.

\subsection{Future work}\label{sec:future}
 
Combined with the previous results in \cite{ACdCdH20} and \cite{ACP21}, the results of this paper nearly complete the analysis of the limiting distributions for one-dimensional RWCRE. The only remaining case to be studied is when $\kappa=1$ where $\kappa$ is the parameter defined in \eqref{kappadef}. 
We leave that case for consideration in a future work, but for now comment on some of the unique difficulties in the case $\kappa = 1$. 
\begin{itemize}
 \item It does not appear that the replacement method will work in the case $\kappa=1$. Indeed, as is seen in the proof of Theorem \ref{thm:mit_mixtures}, the key ingredient needed to implement the replacement method is for the limiting distribution of the RWRE to be upgraded to $L^2$ convergence. However, when $\kappa=1$ the limiting distribution of the RWRE is a $1$-stable random variable which doesn't have finite variance (or mean) and thus one cannot hope for an $L^2$ convergence result for the RWRE.  
 \item A major difficulty in obtaining the limiting distributions for RWCRE when $\kappa \in (1,2]$ is determining the correct scaling factor. Unlike in the cases where the replacement method can be used, the standard deviation $\sqrt{\Var(X_n)}$ isn't always the correct scaling factor for limiting distributions of RWCRE when $\kappa \in (1,2]$ as seen by Theorem \ref{thm:RWCRElim-k2} and the results in \cite{ACP21}. In these cases, a key element in determining the appropriate scaling factor for a limiting distribution of the RWCRE was in obtaining precise asymptotics for the variance of the corresponding RWRE (Corollary \ref{k2-meanvar} and \cite[Theorem 3.8]{ACP21}). Obtaining precise asymptotics for $\Var(Z_n)$ in the case $\kappa=1$ is complicated
 both by the non-linear centering that is necessary for the RWRE limiting distribution in this case and the fact that the precise large deviation estimates from \cite{BD18} do not include this centering term when $\kappa = 1$.
\end{itemize}

\subsection{Notation}\label{sec:notation}

Before continuing on with the rest of the paper, we will introduce here some notation that we will use throughout the remainder of the paper. 

In the description of the models above, in order to more clearly articulate the difference between the models for RWRE and RWCRE we have used the notation $P_0^\mu$ for the annealed law of RWRE and $\P$ for the annealed law of the RWCRE. However, we could expand the measure $\P$ to include copies of RWRE so that equalities in law such as \eqref{basicdec0} become almost sure equalities. We will assume throughout the remainder of the paper that we have done such an expansion of $\P$ and will therefore in a slight abuse of notation also use $\P$ in place of $P_0^\mu$ for the annealed law of a single RWRE. 

Because our main results are stated for the RWCRE centered by its mean, we will often want to use a centered version of the RWRE. Thus, we will use the notation $\tilde{Z}_n = Z_n - \E[Z_n]$. 

Our proofs of tail asymptotics of RWRE in Sections \ref{ssec:proofs_subbalistic} and \ref{sec:kappa2} will use certain facts about regeneration times for RWRE. We recall here the definition of regeneration times for a RWRE as well as some basic facts about regeneration times that we will use in the proofs. 
For a transient RWRE $\{Z_n\}_{n\geq 0}$, the regeneration times $0<R_1<R_2<R_3<\ldots$ are
defined by 
\[
\begin{aligned}
 R_1&=\inf\big\{n>0 \colon \max_{m  < n} Z_{m}< Z_n \leq \min_{m>n} Z_m \big\} \\
 \text{and}\quad 
 R_k &= \inf\big\{n>R_{k-1} \colon \max_{m  < n} Z_{m}< Z_n \leq \min_{m>n} Z_m \big\}, \quad \text{for } k>1.
\end{aligned}
\]
We collect here a few properties of regeneration times that we will
use in our analysis below. Details of these facts can be found in
\cite{SZ99}, \cite[Appendix B]{ACP21}, and the references
therein.  
To state these facts, for convenience of notation we will let $R_0 =
0$ though this is a slight abuse of notation because $R_0$ is not
necessarily a regeneration time (as reflected in the first fact
below). We will be assuming that the distribution $\mu$ on
environments is $\kappa$-regular with $\kappa \in (0,2]$, though most
of these properties are true in greater generality.  
\begin{description}
\item \textbf{I.i.d.\ structure.}  The sequence of joint random
  variables $\{(Z_{R_k}-Z_{R_{k-1}},R_k-R_{k-1}) \}_{k\geq 1}$ are
  independent, and for every $k\geq 2$ the vector
  $(Z_{R_k}-Z_{R_{k-1}},R_k-R_{k-1})$ has the same distribution as
  $(Z_{R_1},R_1)$ under the measure
\[
 \bPP(\cdot) = \P( \, \cdot \mid Z_n\geq 0, \forall n\geq 0 ). 
\]
\item \textbf{Regeneration distances have light tails.} There are
  constants $C,c>0$ such that 
  \begin{equation}
    \label{eq:light_tails_1}
    \P(Z_{R_1} > n) \leq C e^{-cn}.
  \end{equation}
  Note that $\P(Z_{R_2}-Z_{R_1} > n ) = \bPP(Z_{R_1} > n) \leq
  \frac{\P(Z_{R_1} > n)}{\P(Z_n \geq 0, \forall n\geq 0)}$, so that $Z_{R_2}-Z_{R_1}$ also has exponential tails.  
\item \textbf{Regeneration times
    have heavy tails.} There is a constant $C>0$ such that
  \begin{equation}
    \label{eq:reg_times_heavy}
    \P(R_2-R_1 > n) = \bPP(R_1 > n) \sim C n^{-\kappa}.
  \end{equation}
  Under the
  measure $\P$ we have the slightly weaker control on the tail of the
  first regeneration time: $\E[R_1^\gamma] < \infty$ for all
  $\gamma \in (0,\kappa)$.
\item \textbf{Connection with the limiting speed.} If $\kappa > 1$
  then the limiting speed of the RWRE as defined in
  \eqref{speed-limit} is given by  
\begin{equation}\label{vform-reg}
v = \frac{\E[Z_{R_2}-Z_{R_1}]}{\E[R_2-R_1]} = \frac{\bEE[Z_{R_1}]}{\bEE[R_1]}. 
\end{equation}
(Note that \eqref{vform-reg} holds true when $\kappa \in (0,1]$ as
well in the sense that $v=0$ and $\bEE[R_1] = \infty$.) 
\end{description}

\section{RWRE results for the case \texorpdfstring{$\kappa\in(0,1)$}{kappa in (0,1)}}\label{sec:proofs}
In this Section we prove Theorems~\ref{thm:mit_mixtures} and \ref{thm:lp_cov}. 
We will first prove Theorem \ref{thm:lp_cov} as it is the key element of the proof of Theorem \ref{thm:mit_mixtures}. 

\subsection{Proof of Theorem \ref{thm:lp_cov}} \label{ssec:proofs_subbalistic} 
By Theorem 4.6.3 in \cite{Durret}, to prove  Theorem \ref{thm:lp_cov}
it is enough to prove that 
$\{\left|\frac{Z_n}{n^\kappa}\right|^p\}_{n\geq 1}$ is uniformly
integrable for any $p<\infty$, i.e., that
  \[
    \lim_{M\to\infty} \limsup_{n\to\infty} \int_M^\infty p x^{p-1}
    \P(|Z_n| > x n^\kappa) \, \ud x = 0. 
\]
We will obtain bounds on $ \P(|Z_n| > x n^\kappa)$, the tail
probabilities in the integral above by bounding separately the left tails,
$\P(Z_n <- x n^\kappa)$ and the right tails $\P(Z_n > x n^\kappa)$.

\textbf{Left tail bounds.}  
Since $\P(Z_n < -x n^\kappa) = 0$ for $x \geq n^{1-\kappa}$ we have that 
\begin{align*}
 \int_1^\infty p x^{p-1} \P(Z_n < -x n^\kappa) \, \ud x
  &\leq \P(Z_n < -n^\kappa) \int_1^{n^{1-\kappa}} p
    x^{p-1}  \, \ud x \\
    &\leq n^{(1-\kappa)p} \P(Z_n < -n^\kappa), 
\end{align*}
and since \cite[Theorem 1.4]{FGP10} implies that $\P(Z_n < -n^\kappa) = e^{-n^{\kappa + o(1)}}$ this upper bound vanishes as $n\to\infty$. This proves the uniform integrability estimates for the left tails. 

\textbf{Right tail bounds.}
For the right tail bounds we will use regeneration times. 
Note that if the $(m+1)$-st regeneration time occurs after time $n$
then $Z_n \leq Z_{R_{m+1}}$. 
Therefore, for any $m\geq 1$ we have that
\begin{align}
 \P( Z_n > x n^\kappa ) &\leq \P(R_{m+1} < n) + \P(Z_{R_{m+1}} > x n^\kappa) \nonumber \\
 &\leq \bPP(R_m < n) + \P\left(Z_{R_{1}} > \frac{x n^\kappa}{2} \right)
 + \bPP\left(Z_{R_m} > \frac{x n^\kappa}{2} \right), \label{k1-rtub}
\end{align}
where in the last inequality we used that $R_{m+1}-R_1$ and
$Z_{R_{m+1}}-Z_{R_1}$ have the same distribution under $\P$ as do
$R_m$ and $Z_{R_m}$, respectively, under the measure $\bPP$.  By
\eqref{eq:light_tails_1}, there are $C,c >0$ such that
\begin{equation}
  \label{eq:2nd_term}
  \P\left(Z_{R_{1}} > \frac{x n^\kappa}{2} \right) \leq C e^{-c x
    n^\kappa} \leq Ce^{-cx}.
\end{equation}
To bound the first and third terms in \eqref{k1-rtub}
choose $m$ depending on $x$ and $n$ as follows
\begin{equation}\label{k1-mnx}
m = m(x,n) = \left\lfloor \frac{x n^\kappa}{4 \bEE[Z_{R_1}] } \right\rfloor. 
\end{equation}
For the first probability in \eqref{k1-rtub}, by the i.i.d structure
of $\{R_k-R_{k-1}\}_{k\geq 1}$ under $\bPP$,
we obtain that 
$\bPP(R_m < n) 
 \leq 
 \bPP(R_1 < n)^m 
 = \left( 1-\bPP(R_1 \geq n) \right)^m$. 
Now, by \eqref{eq:reg_times_heavy}  and \eqref{k1-mnx},
we obtain that  there is $n_0 \in \bb{N}$ and a constant $c>0$ such
that 
\begin{equation}
  \label{eq:Rm_bound_exp_decay}
  \bPP(R_{m(x,n)} < n) \leq e^{-c x} \text{ for all }n >n_0, x \geq 1.
\end{equation}
For the third probability in \eqref{k1-rtub}, first we use our choice of $m$ in \eqref{k1-mnx} to get that $\bPP\left(Z_{R_m} > \frac{x n^\kappa}{2} \right) 
\leq \bPP\left( \frac{Z_{R_m}}{m} > 2 \bEE[Z_{R_1}] \right)$.
Since $Z_{R_m}$ is the sum of $m$ i.i.d.\ random variables with exponential tails and mean $\bEE[Z_{R_1}]$, 
by Cram\'{e}r's theorem, see \cite[Thm 1.4, p 5]{dHol00},
there is a constant $c>0$ and $n_1 \in \bb{N}$ such that
\begin{equation}
    \label{eq:third_part}
    \bPP\left(Z_{R_m} > \frac{x n^\kappa}{2} \right) 
\leq \bPP\left( \frac{Z_{R_m}}{m} > 2 \bEE[Z_{R_1}] \right)
\leq e^{-cm} \leq
    e^{-c x n^\kappa} \leq e^{-cx} \text{ for all } n>n_1, \, x\geq 1. 
\end{equation}
By \eqref{eq:2nd_term}, \eqref{eq:Rm_bound_exp_decay},  and
\eqref{eq:third_part} we obtain from  \eqref{k1-rtub} that there is $n_2
\in \bb{N}$ and $c>0$ such that 
$\P(Z_n > x n^\kappa) \leq 3 e^{-cx}$ for all $n\geq n_2$ and $x\geq 1$. 
From this it then follows that 
\[
\lim_{M\to\infty} \lim_{n\to\infty} \int_M^\infty p x^{p-1} \P(Z_n > x n^\kappa) \, \ud x = 0. 
\]

%
\subsection{Proof of Theorem
  \ref{thm:mit_mixtures}}\label{ssec:thm:mit_mix}
To prove of Theorem \ref{thm:mit_mixtures}   we follow the ideas in \cite[Section 3]{ACdCdH20} which we now sketch.  Essentially the idea is to use a threshold $J>0$ to distinguish small increments, $T_{k,n} \leq J$, from large increments, $T_{k,n}>J$,  and let this threshold grow after we take $n \to \infty $ to obtain the limit statement. By the CLT for iid random variables, we may replace the small terms by independent copies of Mittag-Leffler distributions as both have the same Gaussian limit. 
Next by using the convergence of \eqref{Znldk1} which holds in $L^2$ we can show that there is a coupling of $(Z^{(k)}_{m}, \Mit^{(k)}, k,m \in \bb{N})$ for which the difference between the increments and copies of Mittag-Leffler can be neglected and we are allowed to \emph{replace}  the left hand side of \eqref{Lawsub} by weighted sums of independent Mittag-Leffler random variables. More explicitly,  recall \eqref{rm-dec} and note that for any $J>0$  we have 
\begin{align*}
\mathfrak{X}_n &:=
\frac{X_n- \bb{E}[X_n]}{\sqrt{\Var(X_n)}} \nonumber \\
&= \sum_k \lambda_{\tau,n}(k) \ind_{T_{k,n} \leq J}\bigg[ \frac{Z^{(k)}_{T_{k,n}} - \bb{E}[Z^{(k)}_{T_{k,n}}]}{\sqrt{\Var(Z^{(k)}_{T_{k,n}})}} \bigg] +\sum_k \lambda_{\tau,n}(k) \ind_{T_{k,n} >J}\bigg[ \frac{Z^{(k)}_{T_{k,n}} - \bb{E}[Z^{(k)}_{T_{k,n}}]}{\sqrt{\Var(Z^{(k)}_{T_{k,n}})}} \bigg],
\end{align*}
For the small increments, the CLT for iid random variables give us that for any $J>0$ and any fixed bounded continuous function $f: \bb{R} \to \bb{R}$ 
\begin{equation}\label{small-replace}
\lim_n \bb{E}\Bigg[f\bigg(\sum_k \lambda_{\tau,n}(k) \ind_{T_{k,n} \leq J}\Big[ \frac{Z^{(k)}_{T_{k,n}} - \bb{E}[Z^{(k)}_{T_{k,n}}]}{\sqrt{\Var(Z^{(k)}_{T_{k,n}})}} \Big]\bigg) \Bigg]
-\bb{E}\Bigg[f\bigg(\sum_k \lambda_{\tau,n}(k) \ind_{T_{k,n} \leq J} \widehat{\mathfrak{M}}_k^{(k)} \bigg)\Bigg] = 0.
\end{equation}
Now, provided the coupling of the random variables ensures almost sure convergence \eqref{Znldk1}, we obtain
\begin{equation}\label{large-replace}
\lim_{J \to \infty}\lim_{n\to \infty} \bb{E}\bigg[\Big(\sum_k \lambda_{\tau,n}(k) \ind_{T_{k,n} >J}\bigg[ \frac{Z^{(k)}_{T_{k,n}} - \bb{E}[Z^{(k)}_{T_{k,n}}]}{\sqrt{\Var(Z^{(k)}_{T_{k,n}})}} - \cMit^{(k)} \bigg] \Big)^2\bigg] = 0.
\end{equation}
A combination of \eqref{small-replace} and  \eqref{large-replace} allows us claim that $\big(\mathfrak{X}_n, n\in \bb{N} \big)$ has the same subsequential weak limits as $\big((\cMit)^{\otimes \lambda_{n}}, n \in \bb{N}\big)$.
The final step in the proof of \eqref{Lawsub} is to identify the subsequential weak limits of $\big((\cMit)^{\otimes \lambda_{n}}, n \in \bb{N}\big)$.
For ease of notation, let $\lambda_j : = \lambda_{\tau,n_j}^\downarrow$, note that $(\cMit)^{\otimes \lambda_j} \overset{(d)}{=} (\cMit)^{\otimes \lambda_{\tau,n_j}} $, and  recall that we are assuming that $\lim_{j\to\infty} \lambda_j(k) = \lambda_*(k)$ for all $k\geq 1$.  Now note that for any $K>0$
\begin{equation}\label{fix-K}
\sum_{k = 1}^K \lambda_{j}(k) \cMit^{(k)} \underset{j\to\infty}{\Longrightarrow} \sum_{k = 1}^K \lambda_{*}(k) \cMit^{(k)}.
\end{equation}
By \eqref{fix-K}, we may take $K_j \to \infty$ slowly enough such that
\[
\sum_{k = 1}^{K_j} \lambda_{j}(k) \cMit^{(k)} \underset{j\to\infty}{\Longrightarrow} \sum_{k = 1}^\infty \lambda_{*}(k) \cMit^{(k)}.
\]
To conclude, we note that we may also take $K_j \to \infty$ slowly enough such that  
\begin{align*}
a(\lambda_*)^2 &= 1 - \sum_{k=1}^\infty \big(\lambda_*(k)\big)^2 = \sum_{k=1}^\infty \big(\lambda_{j}(k)\big)^2 - \sum_{k = 1}^\infty \big(\lambda_{*}(k)\big)^2\\
& = \lim_{j} \sum_{k=1}^\infty \big(\lambda_j(k)\big)^2 - \sum_{k = 1}^{K_j} \big(\lambda_*(k)\big)^2 = \lim_j \sum_{k=1}^\infty \big(\lambda_j(k)\big)^2 - \sum_{k = 1}^{K_j} \big(\lambda_j(k)\big)^2\\
& = \lim_j \sum_{k = K_j+1}^\infty \big(\lambda_j(k)\big)^2,
\end{align*}
and using the Lindeberg condition for triangular arrays, see Theorem 3.4.10 in \cite{Durret},  we obtain that 
\[
\sum_{k = K_j+1}^\infty \lambda_{j}(k) \cMit^{(k)} \underset{j\to\infty}{\Longrightarrow} a(\lambda_*) \Phi,
\]
where $\Phi_0$ is a standard Gaussian random variable. 
By the independence of the sequence $(\Mit^{(k)}, k \in \bb{N})$
we obtain that
\[
\sum_{k = 1}^{\infty} \lambda_{j}(k) \cMit^{(k)}  =  \sum_{k = 1}^{K_j} \lambda_{j}(k) \cMit^{(k)}  +  \sum_{k = K_j + 1}^{\infty} \lambda_{j}(k) \cMit^{(k)} \underset{j\to\infty}{\Longrightarrow} \sum_{k = 1}^\infty \lambda_{*}(k) \cMit^{(k)} + a(\lambda_*) \Phi_0,
\]
where the two terms on the right are independent. This
concludes the proof of \eqref{Lawsub}. To obtain convergence in $L^p$  for all $p<\infty$ it is enough to show that $\sup_n \bb{E}[(\mathfrak{X}_n)^{2r}] < \infty$ for all $r \in \bb{N}$.
This can be proved by representing $\mathfrak{X}_n$ as in \eqref{rm-dec}, using a binomial expansion of $(\mathfrak{X}_n)^{2r}$ and using the following facts: (1) the terms in the decomposition in \eqref{rm-dec} are independent with zero mean, (2) for any $\ell \geq 2$ we have $\sum_{k} (\lambda_{\tau,n}(k))^\ell \leq \sum_{k} (\lambda_{\tau,n} (k))^2 = 1$, and (3) Theorem \ref{thm:lp_cov} implies that $\sup_n \E[(\frac{Z_n - \E[Z_n]}{\sqrt{\Var(Z_n)}})^\ell ] \leq C_\ell < \infty$ for all $\ell <\infty$.

\section{RWRE results for the case \texorpdfstring{$\kappa=2$}{kappa equals 2}}\label{sec:kappa2}

In this section we will prove some of the new RWRE results that will be needed for the analysis of the limiting distributions of RWCRE when the distribution $\mu$ is $2$-regular. 
This will include the proof of Theorem \ref{thm:EZn2lim}, but will also include some new large and moderate deviation tail bounds as well as some asymptotics of truncated moments that will be crucial later in the proof of Theorem \ref{thm:RWCRElim-k2}. 
Since the moments of the RWRE can be expressed in terms of the tails of the distribution of the RWRE, we will need good tail asymptotics of the RWRE to prove Theorem \ref{thm:EZn2lim}. We will divide our analysis of the tails of the RWRE into the right and left tails separately since the asymptotics are very different in either case. 

\subsection{Right tail estimates}

Our main result in this section is the following Gaussian right tail estimate for the RWRE. 

\begin{lemma}
If the distribution $\mu$ is $2$-regular, then there exist constants $C,c>0$ such that for all $n$ sufficiently large, 
\begin{equation}\label{rtail-vn}
\mathbb{P}(Z_n-vn\geq x)\leq Ce^{-c\frac{x^2}{n\log n}}, \qquad \forall x>0. 
\end{equation}
\end{lemma}

\begin{proof}
First of all, note that it is enough to prove the inequality in \eqref{rtail-vn} for $x \in (0,(1-v)n]$ since the probability on the left is zero for $x>(1-v)n$. 

To this end, recall the notation regarding regeneration times introduced in Section \ref{sec:notation} and for any
$x\in (0,(1-v)n]$
let $m=m(x,n)$ be defined by 
\begin{equation}\label{mxdef}
m
= \left \lfloor \frac{vn+\frac{x}{4}}{\bEE[Z_{R_1}]} \right\rfloor. 
\end{equation}
We will use in several places below that
this choice of $m$ implies that there is a constant $c_1>0$ such that $c_1 n \leq m \leq \frac{1}{c_1} n$ for all 
$x\in (0,(1-v)n]$.
For this choice of $m$ we have that for $n$ sufficiently large 
\begin{align}
  &  \mathbb{P}(Z_n-vn\geq x) \nonumber \\
  &\qquad \leq \mathbb{P}(Z_{R_1}\geq x/2)+\mathbb{P}(Z_{R_{m+1}}-Z_{R_1}\geq vn+x/2)+\mathbb{P}(R_{m+1}-R_1<n) \nonumber \\
    &\qquad \leq P(Z_{R_1} \geq x/2) + \bPP\left( Z_{R_{m}}- \bEE[ Z_{R_{m}} ] \geq x/4 \right) + \bPP\left( R_{m} - \bEE[R_{m}] < -x/(8v) \right), \label{decom rtail kappa2}
\end{align}
where in the last inequality we used \eqref{mxdef}, \eqref{vform-reg}, and the i.i.d.\ structure of regeneration times.
We will bound the three terms in \eqref{decom rtail kappa2} separately. 

Since $Z_{R_1}$ has exponential tails, see \eqref{eq:light_tails_1}, the first term in \eqref{decom rtail kappa2} can be bounded by
$Ce^{-cx}$ for $x>0$. 
For the second term in \eqref{decom rtail kappa2}, since 
under $\bPP$ we have that $Z_{R_m}$ is a sum of $m$ i.i.d.\ random variables with exponential tails, 
standard large deviation results 
(e.g. \cite[Chapter III, Theorem 15]{Petrov}) imply that 
there exists a constant $c>0$ such that 
\[
\bPP(Z_{R_m}-\bEE[Z_{R_m}] \geq x/4 )\leq \exp\left\{\frac{-c x^2}{n} \right\}, \qquad \forall x \in (0,(1-v) n].  
\]
(Note that here we are using that $c_1 n \leq m \leq n/c_1$ as noted above.) To bound the last term in \eqref{decom rtail kappa2} we will use the fact that $R_m - \bEE[R_m]$ is a sum of $m$ i.i.d.\ random variables which have mean zero, are bounded below, and have tails decaying like $x^{-2}$ to the right. Thus, applying Corollary \ref{cor:lthtsum} we get for $n$ large enough and $x \leq (1-v)n$ that  
\[
\bPP\left(R_m - \bEE[R_m] < -x/(8v) \right) \leq e^{-c \frac{x^2}{m \log m}} \leq e^{-c \frac{x^2}{n \log n}}. 
\]
(Note that to apply Corollary \ref{cor:lthtsum} we are using both that $x < (1-v)n$ and that $m\geq c_1 n$, while for the second inequality above we are using that $m\leq n/c_1$.)

Putting together the bounds for the three terms in \eqref{decom rtail kappa2}, we have for $n$ large enough that 
$ \mathbb{P}(Z_n-vn\geq x)\leq Ce^{-cx}+e^{-c\frac{x^2}{n}}+e^{-c\frac{x^2}{n\log n}} 
    \leq C e^{-c \frac{x^2}{n\log n}}$ for all $x \in (0,(1-v)n]$. 
\end{proof}

An immediate consequence of \eqref{rtail-vn} is the following bound on the truncated right tail $L^p$ norm of $Z_n-vn$. 
\begin{corollary}[\textbf{Right tail estimates}]\label{cor:rtail}
If the distribution $\mu$ is $2$-regular, then for any $p>0$
\[
    \lim_{M\rightarrow\infty}\limsup_{n\rightarrow\infty}\frac{1}{(n\log n)^{p/2}}\mathbb{E}\left[|Z_n-vn|^p1_{\{Z_n-vn\geq M\sqrt{n\log n}\}}\right]=0.
\]
\end{corollary}

\subsection{Left tail estimates}

The left tail asymptotics of $Z_n-vn$ are much more delicate than the right tail asymptotics.
We can use the annealed large deviation principle to get good bounds on the left tail probabilities $\P(Z_n - vn\leq -x)$, but only when $x \gg nv$ since the annealed large devation rate function is zero on $[0,v]$ (see Theorem \ref{thm:Zn-bor} or \cite{CGZ00}). 
On the other hand, the estimates of Buraczewski and Dyszewski stated in \eqref{prectail-speedcenter}
allow us to get very precise estimates for these probabilities, but only covering $x \in [\sqrt{n}\log^3 n , nv - \log n]$. 
The following is a rougher estimate but allows us to give a useful upper bound for any 
$x \in (0,nv/2]$.


\begin{proposition}[\textbf{General left tail estimates}]\label{general ltail} If the distribution $\mu$ is $2$-regular, then 
there exists a constant $C>0$ such that for $n$ large enough we have
\begin{equation}\label{Zltail}
    \mathbb{P}(Z_n-vn\leq -t\sqrt{n\log n})\leq \frac{C}{t^2(\log n)} + \frac{C}{t^4}, \qquad \text{for } 0< t \leq \frac{v}{2} \sqrt{\frac{n}{\log n}}.
\end{equation}
\end{proposition}

\begin{proof}[Proof of Proposition \ref{general ltail}]
 By taking the constant $C>1$, the bound in \eqref{Zltail} becomes trivial for $t \leq 1/\sqrt{\log n}$. Thus, for the remainder of the proof we will assume that $\frac{1}{\sqrt{\log n}} \leq t \leq \frac{v}{2} \sqrt{\frac{n}{\log n}}$. 
 Letting
 $m = m(n,t) = \fl{ \frac{1}{\bEE[R_1]}(n-t\sqrt{n\log n} ) }$ we have
 \begin{align}
  &\P\left( Z_n - nv \leq -t \sqrt{n\log n} \right) \nonumber \\
  &\leq \P\left( R_{m+1} > n \right) + \P\left( Z_{R_{m+1}} \leq nv - t\sqrt{n\log n} \right) \nonumber \\ 
  &\leq \P\left( R_1 > \frac{t}{2} \sqrt{n\log n} \right) + \bPP\left( R_{m} > n - \frac{t}{2} \sqrt{n \log n}  \right) \nonumber 
  + \bPP\left( Z_{R_{m}} \leq nv - t\sqrt{n\log n} \right) \nonumber \\
  &\leq \P\left( R_1 > \frac{t}{2} \sqrt{n\log n} \right) + \bPP\left( R_m - \bEE[R_m] > \frac{t}{2}\sqrt{n\log n} \right) \label{Zltail12} \\
  &\qquad + \bPP\left( Z_{R_m} - \bEE[ Z_{R_m} ] \leq -\left( \frac{1-v}{2} \right) t \sqrt{n\log n} \right). \label{Zltail3}
 \end{align}
Since 
$\E[R_1^\gamma] < \infty$ for all $\gamma < \kappa = 2$, it follows that 
the first term in \eqref{Zltail12} is bounded by 
$
\frac{2\E[R_1]}{t\sqrt{n\log n}} \leq \frac{v\E[R_1]}{t^2\log n}$ for all $t\leq \frac{v}{2} \sqrt{\frac{n}{\log n}}$. 
For the second term in \eqref{Zltail12}, since  $R_m - \bEE[R_m]$ is the sum of i.i.d.\ terms with tail decay $\bPP(R_1 > x) \sim C x^{-2}$ and since $\frac{m}{n} = \frac{m(n,t)}{n}$ is uniformly bounded away from $0$ and $\infty$ for $t \leq \frac{v}{2}\sqrt{\frac{n}{\log n}}$, by applying Lemma \ref{Snltail} we obtain that 
\[
 \bPP\left( R_m - \bEE[R_m] > \frac{t}{2}\sqrt{n\log n} \right)
 \leq \bPP\left( R_m - \bEE[R_m] > C t \sqrt{m\log m} \right) 
 \leq \frac{C}{t^2 \log n} + \frac{C}{t^4}. 
\]
Finally, since $Z_{R_m}-\bEE[Z_{R_m}]$ is the sum of i.i.d.\ terms with exponential tails, standard large deviation estimates (see \cite[Theorem III.15]{Petrov}) again imply that for $n$ large enough and  $t\leq \frac{v}{2} \sqrt{\frac{n}{\log n}}$ the probability in \eqref{Zltail3} is bounded above by 
\[
 \exp\left\{ -C \frac{t^2 n \log n}{m \vee t \sqrt{n\log n}} \right\} 
 \leq \exp\left\{ -C t^2 \log n \right\} \leq \frac{C}{t^2 \log n}. 
\]
(We are again using here that $m/n$ is bounded away from $0$ and $\infty$.)
 This completes the proof of the Proposition.  
\end{proof}

A consequence of the above left tail estimates for the RWRE is the following lemma which is the key to the proof of Theorem \ref{thm:EZn2lim}. 
\begin{lemma}\label{Z2left}
If the distribution $\mu$ is $2$-regular, then  
\begin{equation}\label{K0vlim}
   \lim_{M\to\infty} \limsup_{n\rightarrow\infty}  \left| \frac{1}{n\log n}\mathbb{E}\left[(Z_n-vn)^21_{\{Z_n-vn\leq -M\sqrt{n\log n}\}}\right] - K_0 v \right| = 0, 
\end{equation}
where $K_0$ is the constant from \eqref{prectail-speedcenter} and $v$ is the limiting speed of the RWRE as in \eqref{speed-limit}. 
\end{lemma}

\begin{proof}
For any fixed $M$,  noting that since $|Z_n|\leq n$,  we have that 
\begin{align*}
& \frac{1}{n\log n}\mathbb{E}\left[|Z_n-vn|^2 1_{\{Z_n-vn\leq -M\sqrt{n\log n}\}}\right] \nonumber \\
 &= M^2 \P( Z_n - vn \leq -M\sqrt{n\log n} ) + \frac{1}{n\log n}\int_{M \sqrt{n\log n}}^{(1+v)n} 2x \, \mathbb{P}(Z_n-vn\leq-x) \,  \ud x.
 \end{align*}
Using Proposition \ref{general ltail} we have that
$M^2 \P( Z_n - vn \leq -M\sqrt{n\log n} ) \leq \frac{C}{\log n} + \frac{C}{M^2}$,
and therefore to prove \eqref{K0vlim} it suffices to 
control the integral term above.
Splitting this integral term into three parts we may write
\begin{equation}\label{3parts}
    \begin{aligned}
    &\frac{1}{n\log n} \int_{M \sqrt{n\log n}}^{(1+v)n} 2x \, \mathbb{P}(Z_n-vn\leq-x) \,  \ud x  \\
    &\qquad = \frac{1}{n\log n} \bigg\{\int_{M \sqrt{n\log n}}^{\sqrt{n}\log^3 n}  +\int_{\sqrt{n}\log^3 n}^{vn-\log n} +\int_{vn-\log n}^{(1+v)n} \bigg\}2x \, \mathbb{P}(Z_n-vn\leq-x) \, \ud x.
    \end{aligned}
\end{equation}
Using the precise tail estimates in \eqref{prectail-speedcenter} one can see that the middle term above converges to $K_0v$ as $n\to\infty$, and thus one needs only to show that the first and third integral terms vanish as $n\to\infty$ and then $M\to\infty$. This can be verified by using the general tail bounds in Proposition \ref{general ltail} for the probabilities in the first integral and by bounding all the probabilities in the third integral by
$\P(Z_n - vn < -vn+\log n) \leq \frac{2 K_0 \log n}{(nv-\log n)^2} \leq C \frac{\log n}{n^2}$, where the upper bound on these probabilities hold for $n$ large by \eqref{prectail-speedcenter}.  
\end{proof}

\subsection{Proof of Theorem \texorpdfstring{\ref{thm:EZn2lim}}{1.11} and centered tail estimates}\label{sec:cente}

Having obtained the left and right tail estimates in the previous two sections, we are now ready to give the proof of Theorem \ref{thm:EZn2lim}. 

\begin{proof}[Proof of Theorem \ref{thm:EZn2lim}]
First of all, note that for $M$ fixed we have
 \begin{align*}
  \left| \frac{\E[(Z_n-nv)^2]}{n\log n} - (b^2+K_0 v) \right|
  &\leq \left| \frac{\E\left[ (Z_n-nv)^2 \ind_{|Z_n-nv| < M\sqrt{n\log n}} \right] }{n\log n} - b^2 E[\Phi^2 \ind_{|\Phi| < M}] \right| \\
  &\qquad + \left| \frac{\E\left[ (Z_n-nv)^2 \ind_{Z_n-nv \leq -M\sqrt{n\log n}} \right] }{n\log n} - K_0 v \right| \\
  &\qquad + b^2 E[\Phi^2 \ind_{|\Phi| \geq M}] 
  + \frac{\E\left[ (Z_n-nv)^2 \ind_{Z_n-nv \geq M\sqrt{n\log n}} \right] }{n\log n}.
\end{align*}
It follows from \eqref{gaus-static-scaling} and the bounded convergence theorem that the first term on the right vanishes as $n\to\infty$ for any fixed $M$. 
Then, it follows from  Corollary \ref{cor:rtail} and Lemma \ref{Z2left} that the last three terms on the right can be made arbitrarily small as $n\to\infty$ by fixing $M$ large enough. This completes the proof of the theorem. 
\end{proof}

The left and right tail asymptotics for the RWRE proved above were for the random walk centered by its limiting speed. However, since results of Theorem \ref{thm:RWCRElim-k2} are proved for the RWCRE centered by its mean, we will need some results on the RWRE centered by its mean as well. 
Since it follows from Corollary \ref{k2-meanvar} that $\E[Z_n] = nv + o(\sqrt{n\log n})$, we can then easily obtain the following analogs of the above results for the RWRE centered by its mean (with different constants and sometimes slightly smaller ranges to which the tail estimates apply). 

\begin{corollary}\label{cor:centered}
Let $(Z_n)_{n\geq 0}$ be a RWRE with a $2$-regular distribution $\mu$ on environments. Then the following right and left tail estimates hold for the centered RWRE.  
\begin{itemize}
 \item \textbf{Right tail estimate:} There exist constants $C,c$ such that for $n$ sufficiently large
\begin{equation}\label{rtail-cen}
\mathbb{P}(Z_n-\mathbb{E}Z_n\geq x) \leq C e^{-\frac{c x^2}{n\log n}}, \quad \forall 
x>0. 
\end{equation}
 \item \textbf{General left tail estimate:} There is a constant $C$ such that for $n$ sufficiently large 
 \begin{equation}\label{cgeneral ltail}
    \mathbb{P}(Z_n-\mathbb{E}Z_n\leq -t\sqrt{n \log n}) \leq \frac{C}{t^2 \log n} + \frac{C}{t^4}, \qquad \text{for } 1 \leq t \leq \frac{v}{2} \sqrt{\frac{n}{\log n}}. 
\end{equation}
 \item \textbf{Precise left tail estimate:} With $K_0$ the same constant as in \eqref{prectail-speedcenter}, we have 
\begin{equation}\label{cprecise ltail}
    \lim_{n\rightarrow\infty}\sup_{\sqrt{n}\log^4n \leq x \leq nv-\sqrt{n \log n}}\Big|\frac{\mathbb{P}(Z_n-\mathbb{E}Z_n\leq -x)}{(nv-x)x^{-2}}-K_0\Big|=0.
\end{equation}
\end{itemize}
\end{corollary}


We close this section with some consequences of the above tail estimates for certain truncated first and second moments of the RWRE which will be needed for our proof of Theorem \ref{thm:RWCRElim-k2}. 
For these lemmas (and throughout Section \ref{sec:k2limdist}), recall from Section \ref{sec:notation} that $\tilde{Z}_n  = Z_n - \E[Z_n]$ is the notation for the centered version of the RWRE.

\begin{lemma}\label{cor-L1tail}
There exists a constant $C>0$ such that for $n$ sufficiently large and $a \geq \sqrt{n}\log^4 n$, 
\begin{equation}\label{L1tail}
\E\left[ |\tilde{Z}_n| \ind_{|\tilde{Z}_n| > a} \right] \leq \frac{Cn}{a}
\end{equation}
\end{lemma}
\begin{proof}
First of all, note that since $|\tilde{Z}_n| \leq 2n$ the bound holds trivially for $a \geq 2n$. On the other hand, one easily sees that if \eqref{L1tail} holds for all $a \in [\sqrt{n}\log^4 n, nv/2]$ then by changing the constant $C$ we get that it also holds for $a \in [nv/2,2n]$. Indeed, if $a \in [nv/2,2n]$ then 
\[
\E\left[ |\tilde{Z}_n| \ind_{|\tilde{Z}_n| > a} \right]
\leq \E\left[ |\tilde{Z}_n| \ind_{|\tilde{Z}_n| > nv/2} \right] 
\leq C \frac{n}{nv/2} \leq \frac{(4C/v) n}{a}. 
\]
Thus, for the remainder of the proof we will assume that $\sqrt{n}\log^4 n \leq a \leq nv/2$. 

To bound the expectation in \eqref{L1tail} we first decompose it as $\E\left[ |\tilde{Z}_n| \ind_{|\tilde{Z}_n| > a} \right] = \E\left[ \tilde{Z}_n \ind_{\tilde{Z}_n > a} \right] + \E\left[ (-\tilde{Z}_n) \ind_{\tilde{Z}_n < -a} \right]$. 
For the right truncated expectation, 
using \eqref{rtail-cen} we get that
\begin{align*}
\E\left[ \tilde{Z}_n \ind_{\tilde{Z}_n > a} \right]
&= a \P(\tilde{Z}_n > a) + \int_a^\infty \P(\tilde{Z}_n > x ) \, \ud x \\
&\leq C a e^{-ca^2/(n\log n)} + C \frac{n\log n}{a} e^{-c a^2/(n\log n)}.  
\end{align*}
Since $a\geq \sqrt{n} \log^4 n$ implies that $e^{-ca^2/(n\log n)} \leq e^{-c \log^7 n}$ and since $a \leq nv/2$, it follows that the above bound is less than $\frac{C'n}{a}$ for some $C'<\infty$ for all $n$ sufficiently large and $a \in [\sqrt{n}\log^4 n, nv/2]$.  
For the left truncated expectation, 
using the assumption that $\sqrt{n}\log^4 n \leq a \leq nv/2$ and the precise tail bounds in \eqref{cprecise ltail} we obtain for $n$ sufficiently large that
\begin{align*}
\E\left[ (-\tilde{Z}_n) \ind_{\tilde{Z}_n < -a} \right]
&= a \P( \tilde{Z}_n < -a ) +  \int_a^{2n} \P\left( \tilde{Z}_n < -x \right) \,  \ud x \\
& \leq a \P(\tilde{Z}_n < -a ) + \int_a^{nv-\sqrt{n\log n}} \P\left( \tilde{Z}_n < -x \right) \, \ud x \\
&\qquad + 2n \P\left( \tilde{Z}_n < -(nv-\sqrt{n\log n}) \right) \\
&\leq  \frac{2K_0 v n}{a} + \int_a^{nv-\sqrt{n\log n}} \frac{2K_0 v n}{x^2}  \, \ud x + \frac{8 K_0 \sqrt{\log n} }{v^2\sqrt{n} } \\
&\leq \frac{4K_0 v n}{a} + \frac{8 K_0 \sqrt{\log n} }{v^2 \sqrt{n}}. 
\end{align*}
Finally, note that $a\leq nv/2$ implies that $\frac{\sqrt{\log n}}{\sqrt{n}} \leq \frac{n}{a}$ for $n$ sufficiently large. 
\end{proof}

\begin{lemma}\label{cor-L2trunc}
 If $\sqrt{n \log n} \leq a \leq \sqrt{n}\log^4 n$, then 
 \[
 \E\left[ (\tilde{Z}_n)^2 \ind_{a < |\tilde{Z}_n| \leq \sqrt{n}\log^4 n} \right] \leq C n \log \log n + C \frac{n^2 \log^2 n}{a^2}. 
 \]
\end{lemma}
\begin{proof}
We begin by noting that 
\begin{equation}\label{tL2dec}
\E\left[ (\tilde{Z}_n)^2 \ind_{a < |\tilde{Z}_n| \leq \sqrt{n}\log^4 n} \right] 
\leq a^2 \P( |\tilde{Z}_n| > a ) + \int_a^{\sqrt{n}\log^4 n} 2x \P( |\tilde{Z}_n| > x ) \, \ud x. 
\end{equation}
Using \eqref{rtail-cen} and \eqref{cgeneral ltail} we can bound the first term on the right in \eqref{tL2dec} by
\[
a^2 \P( |\tilde{Z}_n| > a ) \leq C a^2 e^{-\frac{c a^2}{n\log n}} + Cn + C  \frac{n^2\log^2 n}{a^2} 
\leq Cn + C \frac{n^2\log^2 n}{a^2}, 
\]
where the last inequality is justified by noting that (by elementary calculus)
\begin{equation}\label{bound_a2exp}
 a^2 e^{-\frac{ca^2}{n\log n}} = \frac{1}{a^2}\left(a^4 e^{-\frac{ca^2}{n\log n}} \right) \leq \frac{1}{a^2}\left( \frac{4}{e^2 c^2} n^2 \log^2 n \right).   
\end{equation}
For the integral term in \eqref{tL2dec}, again using \eqref{rtail-cen} and \eqref{cgeneral ltail}, we obtain that 
\begin{align*}
\int_a^{\sqrt{n}\log^4 n} 2x \P( |\tilde{Z}_n| > x ) \, \ud x
&\leq C \int_a^{\sqrt{n}\log^4 n} \left\{ x e^{-\frac{c x^2}{n\log n}} + \frac{n}{x} + \frac{n^2\log^2 n}{x^3} \right\} \, \ud x \\
&\leq C \left\{ n\log n \, e^{-\frac{c a^2}{n\log n}} + n \log \left( \frac{\sqrt{n} \log^4 n}{a} \right) + \frac{n^2 \log^2 n}{a^2} \right\} \\
&\leq C \left\{ n \log \log n + \frac{n^2 \log^2 n}{a^2} \right\}, 
\end{align*}
where in the last inequality we used \eqref{bound_a2exp} and that $a\geq \sqrt{n\log n}$ implies $\log\left( \frac{\sqrt{n}\log^4 n}{a} \right) \leq \frac{7}{2} \log \log n$. 
\end{proof}

\begin{lemma}\label{cor-smalltsm}
If $b$ is the scaling constant from the limiting distribution in \eqref{gaus-static-scaling}, then 
\begin{equation}\label{tv-tsm-small}
\lim_{n\to\infty} \frac{\Var\left(\tilde{Z}_n \ind_{|\tilde{Z}_n| \leq \sqrt{n}\log^4 n} \right)}{n\log n} = 
\lim_{n\to\infty} \frac{\E\left[ (\tilde{Z}_n)^2 \ind_{|\tilde{Z}_n| \leq \sqrt{n}\log^4 n} \right]}{n\log n}
= b^2. 
\end{equation}
\end{lemma}
\begin{proof}
First of all, note that 
\begin{align*}
&\left| \Var\left(\tilde{Z}_n \ind_{|\tilde{Z}_n| \leq \sqrt{n}\log^4 n} \right) 
- \E\left[ (\tilde{Z}_n)^2 \ind_{|\tilde{Z}_n| \leq \sqrt{n}\log^4 n} \right] \right| \\
&\qquad = \E\left[ \tilde{Z}_n \ind_{|\tilde{Z}_n| \leq \sqrt{n}\log^4 n} \right]^2 
= \E\left[ \tilde{Z}_n \ind_{|\tilde{Z}_n| > \sqrt{n}\log^4 n} \right]^2
\leq \frac{C n}{\log^8 n},
\end{align*}
where in the second equality we used that $\E[\tilde{Z}_n] = 0$ and in the last inequality we used Lemma \ref{cor-L1tail}. 
Since this bound is $o(n\log n)$, it is enough to only prove the second equality in \eqref{tv-tsm-small}. 
To this end, fix $M>1$ and note that 
\begin{align*}
 \left| \frac{\E\left[ (\tilde{Z}_n)^2 \ind_{|\tilde{Z}_n| \leq \sqrt{n}\log^4 n} \right]}{n\log n} - b^2 \right|
 &\leq \left| \frac{\E\left[ (\tilde{Z}_n)^2 \ind_{|\tilde{Z}_n| \leq M\sqrt{n\log n} } \right]}{n\log n} - b^2 E[\Phi^2 \ind_{|\Phi|\leq M} ] \right| \\
&\qquad  + \frac{\E\left[ (\tilde{Z}_n)^2 \ind_{M\sqrt{n\log n} < |\tilde{Z}_n| \leq \sqrt{n}\log^4 n} \right]}{n\log n} + b^2 E[\Phi^2 \ind_{|\Phi|>M} ]. 
\end{align*}
The first term on the right vanishes by the bounded convergence theorem (together with \eqref{gaus-static-scaling}  and Corollary \ref{k2-meanvar}), while the second term on the right can be bounded by $C\frac{\log\log n}{\log n} + \frac{C}{M^2}$ by Lemma \ref{cor-L2trunc}. Therefore, taking first $n\to\infty$ and then $M\to\infty$ completes the proof. 
\end{proof}

\section{Transient RWCRE when \texorpdfstring{$\kappa=2$}{kappa equals 2}: Arbitrary cooling}
\label{sec:k2limdist}

In this section we will give the proof of Theorem \ref{thm:RWCRElim-k2}. 
We will prove that the limiting distribution is Gaussian by proving that for any subsequence $n_j\to \infty$ there is a further subsequence $n_{j_k}$ along which $\frac{ X_{n_{j_k}} - \E[ X_{n_{j_k}} ] }{\beta_{n_{j_k}} s_{n_{j_k}} } \Rightarrow \Phi$. Therefore, at several times throughout the proof we will be able to assume that some nice additional property holds by passing to a subsequence along which it is true.

Before starting the proof of Theorem \ref{thm:RWCRElim-k2} we will introduce some notation and give a brief idea of the proof. 
Our starting point for the analysis of the distribution of $X_n$ is the decomposition into sums of increments of independent copies of RWRE as in \eqref{basicdec}. 
We will analyze this sum by separating the terms in the sum into the ``large" and the ``small" terms. The small terms will be those whose variance makes up a negligible fraction of the total variance of the sum, and the large terms will be the remaining terms. To be more precise, let
\[
s_n^2 = \Var(X_{n}) = \sum_{k=1}^{\ell_n+1} \sigma_{k,n}^2, 
\qquad \text{where}\quad \sigma_{k,n}^2 = \Var(Z_{T_{k,n}}), 
\]
and then for any $\delta>0$ and $n\geq 1$
let 
\begin{align*}
 I_{n,\delta}^+ &=  \left\{ k\in \{1,2,\ldots,\ell_n+1\}: \,  \sigma_{k,n}^2 > \delta s_n^2 \right\}, \\
 \text{and}\quad 
  I_{n,\delta}^- &=  \left\{k\in \{1,2,\ldots,\ell_n+1\}: \,  \sigma_{k,n}^2 \leq \delta s_n^2\right\}. 
\end{align*}
Finally, letting
\[
\mathcal{Z}_{n,\delta}^+ = \sum_{k \in I_{n,\delta}^+} (Z_{T_{k,n}}^{(k)} - \E[Z_{T_{k,n}}])
 \quad\text{and}\quad 
\mathcal{Z}_{n,\delta}^- = \sum_{k \in I_{n,\delta}^-} (Z_{T_{k,n}}^{(k)} - \E[Z_{T_{k,n}}]) 
\]
we can rewrite $X_n - \E[X_n] = \mathcal{Z}_{n,\delta}^+ + \mathcal{Z}_{n,\delta}^-$. 
The main idea is then to show that, after appropriate scaling, both $Z_{n,\delta}^+$ and $Z_{n,\delta}^-$ are approximately Gaussian (unless the set $I_{n,\delta}^+$ or $I_{n,\delta}^-$ is empty in which case one only needs that the non-zero term is approximately Gaussian). 
Unfortunately, this argument doesn't quite work for a fixed $\delta$ as our analysis of the small terms will actually require $\delta$ to vanish as $n\to\infty$. 
That is, we will show that there exists a sequence $\delta_n \to 0$ such that both $\mathcal{Z}_{n,\delta_n}^+$ and $\mathcal{Z}_{n,\delta_n}^-$ converge in distribution to Gaussians when properly scaled. Since $\mathcal{Z}_{n,\delta_n}^+$ and $\mathcal{Z}_{n,\delta_n}^-$ are independent this will then imply that $X_n - \E[X_n]$ converges to a Gaussian when properly scaled.
Finally, we note that as part of our proof we will show that the proper scaling ends up differing from $s_n = \sqrt{\Var(X_n)}$ by a multiplicative factor which asymptotically lies in $[\beta,1]$, where the constant $\beta$ is defined in \eqref{Zn-limdist-st}.


Having outlined the general strategy, we will now begin proving some lemmas which complete the main steps of the proof. 

\noindent\textbf{Gaussian convergence for the large parts.}
The first step is to prove the convergence of the sums over only the large cooling intervals. 
We cannot simply claim that $\mathcal{Z}_{n,\delta}^+ \Rightarrow \Phi$ as $n\to\infty$ since it could be that  $I_{n,\delta}^+ = \emptyset$ for infinitely many $n$. 
However, the following lemma show that $\mathcal{Z}_{n,\delta}^+$ is approximately Gaussian whenever the set $I_{n,\delta}^+$ is not empty. 

\begin{lemma}\label{lem-step1}
 For any $\delta>0$ and $n\geq 1$ let $s_{n,\delta,+}^2 := \Var(\mathcal{Z}_{n,\delta}^+)$. If $n_j\to\infty$ is a subsequence such that $s_{n_j,\delta,+}>0$ (or equivalently $I_{n_j,\delta}^+ \neq \emptyset$) for all large $j$ then 
 \begin{equation}\label{bigparts}
 \frac{\mathcal{Z}_{n_j,\delta}^+}{s_{n_j,\delta,+}} \underset{j\to\infty}{\Longrightarrow} \beta \Phi. 
 \end{equation}
\end{lemma}

\begin{proof}
We begin by noting that if $s_{n_j,\delta,+}>0$ then we can write 
\[
 \frac{\mathcal{Z}^+_{n_j,\delta}}{s_{n_j,\delta,+}}
 = \sum_{k \in I_{n_j,\delta}^+} \frac{\sigma_{k,n_j}}{s_{n_j,\delta,+}} \frac{\tilde{Z}_{T_{k,n_j}}^{(k)}}{\sigma_{k,n_j}}.
\]
Since the set $I_{n,\delta}^+$ can contain at most $\frac{1}{\delta}$ elements,  
by passing to a subsequence if needed we can assume without loss of generality that there is an integer $1\leq m \leq \frac{1}{\delta}$ such that $|I_{n_j,\delta}^+| = m$ for all $j$; that is, the sum in the decomposition has exactly $m$ terms.
It then follows from the limiting distribution for the RWRE in \eqref{Zn-limdist-st}
that $\{\tilde{Z}^{(k)}_{T_{k,n_j}} / \sigma_{k,n_j} \}_{k \in I_{n_j,\delta}^+} \Rightarrow \{\beta \Phi_i \}_{1\leq i\leq m}$, where the $\Phi_i$ are i.i.d.\ standard Gaussian random variables.\footnote{Note here that we are using that $T_{k,n}$ is large if $k \in I_{n,\delta}^+$. Indeed, it follows from Corollary \ref{k2-meanvar} and the definition of $I_{n,\delta}^+$ that $k \in I_{n,\delta}^+$ implies that $T_{k,n}\log (T_{k,n}) \geq C \delta s_n^2$ for some $C>0$. 
}
The claimed convergence in \eqref{bigparts} then follows from this, together with the observation that $\sum_{k \in I_{n_j,\delta}^+} (\frac{\sigma_{k,n_j}}{s_{n_j,\delta,+}})^2 = 1$. 
\end{proof}

\begin{corollary}\label{cor-step2}
There exists a sequence $\delta_n \to 0$ such that either (1) $s_{n,\delta_n}^+ = 0$ for all $n$ large, or (2) for any subsequence $n_j\to\infty$ such that $s_{n_j,\delta_{n_j},+}>0$ for all but finitely many $j$ we have 
\[
\frac{\mathcal{Z}^+_{n_j,\delta_{n_j}}}{s_{n_j,\delta_{n_j},+}} \underset{j\to\infty}{\Longrightarrow} \beta \Phi.
\]
\end{corollary}

\begin{proof}
 Let $d$ be a metric on the space of Borel probability measures on $\mathbb{R}$ that is consistent with the topology of weak convergence (e.g., the Levy metric), and in a slight abuse of notation for random variables $X$ and $Y$ with distributions $\mu_X$ and $\mu_Y$, respectively, we will use $d(X,Y)$ in place of $d(\mu_X,\mu_Y)$. With this notation, for any $n\geq 1$ and $\delta>0$ we can let 
\[
 \e_{n,\delta} = 
 \begin{cases}
  d\left( \frac{\mathcal{Z}_{n,\delta}^+}{s_{n,\delta,+}},  \beta \Phi \right) & \text{if } s_{n,\delta,+}>0 \\
  0 & \text{if } s_{n,\delta,+} = 0. 
 \end{cases}
\]
It follows from Lemma \ref{lem-step1} that $\lim_{n\to\infty} \e_{n,\delta} = 0$, for any fixed $\delta>0$. 
Let $m_0 = 0$ and define integers $m_j$ for $j\geq 1$ inductively as follows. 
Let $m_j > m_{j-1}$ be such that $\e_{n,1/j} < 1/j$ for all $n>m_j$. Let $\delta_n = 1$ for $n< m_1$ and if $m_j \leq n < m_{j+1}$ for some $j\geq 1$ then $\delta_n = 1/j$. 
For this sequence $\delta_n$ we have $\lim_{n\to\infty} \e_{n,\delta_n} = 0$, and this completes the proof of the corollary.  
\end{proof}

Let the sequence $\delta_n \to 0$ be chosen as in Corollary \ref{cor-step2}.
In general we would like to prove that $\mathcal{Z}_{n,\delta_n}^-$ properly scaled also converges to a centered Gaussian, but in fact this isn't necessarily true if 
$\Var(\mathcal{Z}_{n,\delta_n}^-)$
doesn't grow to $\infty$. However, the following corollary shows that in this case we already have the limiting distribution we are after because only the large parts will be relevant in the limit. 

\begin{corollary}\label{cor-step3}
For any $\delta>0$ let $s_{n,\delta,-} := \Var(\mathcal{Z}_{n,\delta}^-)$.  
If $\delta_n$ is the sequence from Corollary \ref{cor-step2} and $n_j\to \infty$ is a subsequence such that
$\lim_{j\to\infty} \frac{s_{n_j,\delta_{n_j},-}}{s_{n_j}} = 0$, 
then 
\[
\frac{ X_{n_j} - \E[X_{n_j}] }{ \beta s_{n_j} } \underset{j\to\infty}{\Longrightarrow} \Phi. 
\]
\end{corollary} 

\begin{proof}
Since $s_n^2 = s_{n,\delta_n,-}^2 + s_{n,\delta_n,+}^2$, the assumption that $s_{n_j,\delta_{n_j},-}/s_{n_j} \to 0$ implies that 
\begin{equation}\label{bigall}
\lim_{j\to \infty} \frac{s_{n_j,\delta_{n_j},+}}{s_{n_j}} = 1. 
\end{equation}
Next we write the decomposition
\[
\frac{ X_{n_j} - \E[X_{n_j}] }{ \beta s_{n_j} }
= \frac{\mathcal{Z}_{n_j,\delta_{n_j}}^-}{\beta s_{n_j}} + \frac{\mathcal{Z}_{n_j,\delta_{n_j}}^+}{\beta s_{n_j}}.
\]
The variance of the first term on the right vanishes as $j\to \infty$ by the assumption on the subsequence $n_j$ in the statement of the corollary, and the second term converges in distribution to a standard Gaussian by \eqref{bigall} and Corollary \ref{cor-step2}. 
\end{proof}

\noindent\textbf{Gaussian convergence for the small parts.}
Due to Corollary \ref{cor-step3}, 
we will only need to consider the limiting distributions $X_n$ along subsequences $n_j$ such that 
\begin{equation}\label{smrat}
\liminf_{j\to\infty} \frac{s_{n_j,\delta_{n_j},-}}{s_{n_j}} = \theta \in (0,1]. 
\end{equation}
The following lemma shows that under this assumption $\mathcal{Z}_{n_j,\delta_{n_j}}^-$ converges to a Gaussian, but determining the proper scaling to get a standard Gaussian limit is quite delicate.

\begin{lemma}\label{lem-smallparts}
Let $\delta_n \to 0$ and assume that $n_j$ is a subsequence such that \eqref{smrat} holds. Then, 
\[
 \frac{\mathcal{Z}_{n_j,\delta_{n_j}}^-}{\tilde{s}_{n_j,\delta_{n_j}}}
 \Rightarrow \Phi,  
\]
where for any $n\geq 1$ and $\delta>0$
\begin{equation}\label{tsndn}
 \tilde{s}_{n,\delta}^2 = \Var\left( \sum_{k \in I_{n,\delta}^-} \tilde{Z}_{T_{k,n}}^{(k)} \ind_{A_{k,n}} \right) 
 \text{ and } 
 A_{k,n} = \left\{ |\tilde{Z}_{T_{k,n}}^{(k)}|\leq \frac{s_n}{\sqrt{\log s_n}} \vee \sqrt{T_{k,n}}\log^4 T_{k,n} \right\}.
\end{equation}
\end{lemma}

\begin{proof}
For simplicity of notation, we will give the proof under the assumption that \eqref{smrat} holds without taking a subsequence. That is, 
\begin{equation}\label{smrat-n}
\liminf_{n\to\infty} \frac{s_{n,\delta_n,-}}{s_n} = \theta \in (0,1]. 
\end{equation}
We begin by decomposing 
\begin{align}
\frac{\mathcal{Z}_{n,\delta_n}^-}{\tilde{s}_{n,\delta_n}}
&= \frac{1}{\tilde{s}_{n,\delta_n}} \sum_{k \in I_{n,\delta_n}^-} \left( \tilde{Z}_{T_{k,n}}^{(k)} \ind_{A_{k,n}} - \E\left[ \tilde{Z}_{T_{k,n}}^{(k)} \ind_{A_{k,n}} \right] \right) \label{Gpart} \\
&\qquad + \frac{1}{\tilde{s}_{n,\delta_n}} \sum_{k \in I_{n,\delta_n}^-} \left( \tilde{Z}_{T_{k,n}}^{(k)} \ind_{A_{k,n}^c} - \E\left[ \tilde{Z}_{T_{k,n}}^{(k)} \ind_{A_{k,n}^c} \right] \right). \label{negpart}
\end{align}
We will prove that \eqref{Gpart} converges in distribution to a standard Gaussian while \eqref{negpart} converges in distribution to zero.

\noindent\textbf{Proof that \eqref{negpart} is negligible.} 
We will show that \eqref{negpart} converges to zero in $L^1$. 
To this end, note first of all that 
\[
\E\left[ \left| \sum_{k \in I_{n,\delta_n}^-} \left( \tilde{Z}_{T_{k,n}}^{(k)} \ind_{A_{k,n}^c} - \E\left[ \tilde{Z}_{T_{k,n}}^{(k)} \ind_{A_{k,n}^c} \right] \right) \right| \right] 
\leq 2 \sum_{k \in I_{n,\delta_n}^-} \E\left[ |\tilde{Z}_{T_{k,n}}| \ind_{A_{k,n}^c} \right]. 
\]
Since $|\tilde{Z}^{(k)}_{T_{k,n}}| \leq 2T_{k,n}$, the event $A_{k,n}^c$ is empty if $2T_{k,n} < \frac{s_n}{\sqrt{\log s_n}}$. Therefore, we need only consider the terms in the sum where $T_{k,n} \geq \frac{s_n}{2\sqrt{\log s_n}} \geq \sqrt{s_n}$. For these terms we can use Lemma \ref{cor-L1tail} to bound the sum by 
\begin{align*}
\sum_{ k \in I_{n,\delta_n}^- } \E\left[ |\tilde{Z}_{T_{k,n}}| \ind_{A_{k,n}^c} \right] 
&\leq C \sum_{ \substack{k \in I_{n,\delta_n}^- \\ T_{k,n} \geq \sqrt{s_n}} } \frac{T_{k,n}}{ \left( \frac{s_n}{\sqrt{\log s_n}} \right) } \\
&\leq C \frac{\sqrt{\log s_n}}{s_n} \sum_{ \substack{k \in I_{n,\delta_n}^- \\ T_{k,n} \geq \sqrt{s_n}} } \frac{\sigma_{k,n}^2}{\log T_{k,n}} 
\leq \frac{C}{s_n\sqrt{\log s_n}} \sum_{k=1}^{\ell_n+1} \sigma_{k,n}^2
= C \frac{s_n}{\sqrt{\log s_n}},  
\end{align*}
where in the second inequality we used that Corollary \ref{k2-meanvar} implies that there is a constant $C$ such that $T_{k,n} \log T_{k,n} \leq C \sigma_{k,n}^2$, and in the third inequality we used that $T_{k,n} \geq \sqrt{s_n}$ implies that $\log T_{k,n} \geq (1/2) \log s_n$. 
Thus, to complete the proof that \eqref{negpart} converges to zero in $L^1$, we need only to show that $\lim_{n\to\infty} \frac{s_n}{\tilde{s}_{n,\delta_n} \sqrt{\log s_n}} = 0$. 
In fact, we will show that 
\begin{equation}\label{tsndnrat}
 \liminf_{n\to\infty} \frac{\tilde{s}_{n,\delta_n}}{s_{n,\delta_n,-} } \geq \beta, 
\end{equation}
which combined with our assumption \eqref{smrat-n} is enough to show that 
$\lim_{n\to\infty} \frac{s_n}{\tilde{s}_{n,\delta_n} \sqrt{\log s_n}} = 0$.
To prove \eqref{tsndnrat}, note first of all that 
$\tilde{s}_{n,\delta_n}^2 = \sum_{k \in I_{n,\delta_n}^-} \Var\left( \tilde{Z}_{T_{k,n}}^{(k)} \ind_{A_{k,n}} \right)$
and that 
$s_{n,\delta_n,-}^2 = \sum_{k \in I_{n,\delta_n}^-} \sigma_{k,n}^2$
so that it is enough to show that 
\begin{equation}\label{tVarlb1}
\liminf_{n\to \infty} \inf_{k \in I_{n,\delta_n}^-} \frac{ \Var\left( \tilde{Z}_{T_{k,n}}^{(k)} \ind_{A_{k,n}} \right) }{\sigma_{k,n}^2} \geq \beta^2. 
\end{equation}
To this end, first note that 
$\Var\left( \tilde{Z}_{T_{k,n}}^{(k)} \ind_{A_{k,n}} \right) = \sigma_{k,n}^2$ 
if $T_{k,n} \leq \sqrt{s_n}$ since as noted above 
$\ind_{A_{k,n}} \equiv 1$ 
in this case.
Thus, we need only to get a good bound on $\Var\left( \tilde{Z}_{T_{k,n}}^{(k)} \ind_{A_{k,n}} \right)$ when $T_{k,n} \geq \sqrt{s_n}$. 
For this, note that 
\begin{align}
\Var\left( \tilde{Z}_{T_{k,n}}^{(k)} \ind_{A_{k,n}} \right) 
&= \E\left[ (\tilde{Z}_{T_{k,n}}^{(k)})^2 \ind_{A_{k,n}} \right] - \E\left[ \tilde{Z}_{T_{k,n}}^{(k)} \ind_{A_{k,n}} \right]^2 \nonumber \\
&= \E\left[ (\tilde{Z}_{T_{k,n}}^{(k)})^2 \ind_{A_{k,n}} \right] - \E\left[ \tilde{Z}_{T_{k,n}}^{(k)} \ind_{A_{k,n}^c} \right]^2 \nonumber \\
&\geq \E\left[ (\tilde{Z}_{T_{k,n}})^2 \ind_{|\tilde{Z}_{T_{k,n}}| \leq \sqrt{T_{k,n}} \log^4 T_{k,n}} \right] - \E\left[ |\tilde{Z}_{T_{k,n}}| \ind_{|\tilde{Z}_{T_{k,n}}| > \sqrt{T_k} \log^4 T_{k,n}} \right]^2 \nonumber \\
&\geq \E\left[ (\tilde{Z}_{T_{k,n}})^2 \ind_{|\tilde{Z}_{T_{k,n}}| \leq \sqrt{T_{k,n}} \log^4 T_{k,n}} \right] - C \frac{T_{k,n}}{\log^8 T_{k,n}}, \label{tVarlb2}
\end{align}
where the second equality follows from the fact that $\E\left[ \tilde{Z}_{T_k}^{(k)} \ind_{A_{k,n}} \right] = - \E\left[ \tilde{Z}_{T_k}^{(k)} \ind_{A_{k,n}^c} \right]$ since $\E[ \tilde{Z}_{T_k} ] = 0$, and the last inequality follows from Lemma \ref{cor-L1tail}. 
Recall that we only need to use the lower bound \eqref{tVarlb2} when $T_k \geq \sqrt{s_n}$ and note that \eqref{eq:2_mean_var_scaling} and Lemma \ref{cor-smalltsm}  imply that
\[
\lim_{n\to\infty} \frac{\E\left[ (\tilde{Z}_n)^2 \ind_{|\tilde{Z}_n| \leq \sqrt{n}\log^4 n} \right]}{\Var(Z_n)} 
= \beta^2.
\]
This completes the proof of \eqref{tVarlb1} and thus also the proof that \eqref{negpart} converges to zero in $L^1$. 

\noindent\textbf{Proof of convergence of \eqref{Gpart}.}
To prove the convergence of \eqref{Gpart} to a standard Gaussian we will use the Lindeberg-Feller CLT.
Since the normalization $\tilde{s}_{n,\delta_n}$ is chosen so that $\eqref{Gpart}$ has variance $1$ and since assumption \eqref{smrat-n} together with \eqref{tsndnrat} implies that $\tilde{s}_{n,\delta_n} \to \infty$, we need
only to check the Lindeberg condition; that is, for any $\e>0$
{\small
\[
 \lim_{n\to\infty} \frac{1}{\tilde{s}_{n,\delta_n}^2} 
 \sum_{k \in I_{n,\delta_n}^-} \E\left[ \left( \tilde{Z}_{T_{k,n}}^{(k)} \ind_{A_{k,n}} - \E\left[ \tilde{Z}_{T_{k,n}}^{(k)} \ind_{A_{k,n}} \right] \right)^2 \ind_{\left| \tilde{Z}_{T_{k,n}}^{(k)} \ind_{A_{k,n}} - \E\left[ \tilde{Z}_{T_{k,n}}^{(k)} \ind_{A_{k,n}} \right] \right| > \e\tilde{s}_{n,\delta_n} } \right] = 0. 
\]
}%
However, by \eqref{smrat-n} and \eqref{tsndnrat} it is enough to prove the above statement with $\tilde{s}_{n,\delta_n}$ replaced by $s_n$. That is, we need to show for any $\e>0$ that
\begin{equation}\label{Lind}
 \lim_{n\to\infty} \frac{1}{s_n^2} 
 \sum_{k \in I_{n,\delta_n}^-} \E\left[ \left( \tilde{Z}_{T_{k,n}}^{(k)} \ind_{A_{k,n}} - \E\left[ \tilde{Z}_{T_{k,n}}^{(k)} \ind_{A_{k,n}} \right] \right)^2 \ind_{\left| \tilde{Z}_{T_{k,n}}^{(k)} \ind_{A_{k,n}} - \E\left[ \tilde{Z}_{T_{k,n}}^{(k)} \ind_{A_{k,n}} \right] \right| > \e s_n } \right] = 0. 
\end{equation}
To obtain a simple bound on the expectation above, note that 
\[
E\left[ (Y-\mu)^2 \ind_{|Y-\mu|>a} \right] \leq 2E\left[Y^2 \ind_{|Y|>a/2}\right] + 2\mu^2, \quad \text{if } |\mu|<\frac{a}{2}. 
\]
To apply this simple bound to the expectations in \eqref{Lind} we need to check that 
\begin{equation}\label{tmeancheck}
 \E\left[ \tilde{Z}_{T_{k,n}}^{(k)} \ind_{A_{k,n}} \right] \leq \frac{\e s_n}{2}, \qquad \forall k \in I_{n,\delta_n}^-. 
\end{equation}
It follows from Lemma \ref{cor-L1tail} that $\E\left[ \tilde{Z}_{T_{k,n}}^{(k)} \ind_{A_{k,n}} \right] \leq C \frac{T_{k,n}}{s_n/\sqrt{\log s_n}}$. 
Also, note that by Corollary \ref{k2-meanvar} there is a constant $C$ such that $k \in I_{n,\delta_n}^-$ implies that $T_{k,n} \log T_{k,n} \leq C \sigma_{k,n}^2 \leq C \delta_n s_n^2$, and thus $C \frac{T_{k,n}}{s_n/\sqrt{\log s_n}} \leq \frac{\e s_n}{2}$
 for all $n$ large enough and $k \in I_{n,\delta_n}^-$. 
This completes the verification of \eqref{tmeancheck}, and thus to check \eqref{Lind} it is enough to prove for all $\e>0$ that
\begin{equation}\label{Lind2}
 \lim_{n\to\infty} \frac{1}{s_n^2} \sum_{k \in I_{n,\delta_n}^-}
\left\{ 
\E\left[ (\tilde{Z}_{T_{k,n}}^{(k)})^2 \ind_{A_{k,n} \cap \{|\tilde{Z}_{T_{k,n}}^{(k)}| > \e s_n \} } \right]  
+ \E\left[ \tilde{Z}_{T_{k,n}}^{(k)} \ind_{A_{k,n}} \right]^2 
 \right\} = 0.
\end{equation} 
For the first expectation inside the sum in \eqref{Lind2}, note that 
for $n$ sufficiently large (so that $1/\sqrt{\log s_n} < \epsilon$) we have
\begin{align*}
A_{k,n} \cap \left\{|\tilde{Z}_{T_{k,n}}^{(k)}| > \e s_n \right\} 
&= \left\{ \e s_n < |\tilde{Z}_{T_{k,n}}^{(k)}| \leq \frac{s_n}{\sqrt{\log s_n}} \vee \sqrt{T_{k,n}} \log^4 T_{k,n} \right\} \\
&= \left\{ \e s_n < |\tilde{Z}_{T_{k,n}}^{(k)}| \leq \sqrt{T_{k,n}} \log^4 T_{k,n} \right\},  
\end{align*}
and thus
\[
\E\left[ (\tilde{Z}_{T_{k,n}}^{(k)})^2 \ind_{A_{k,n} \cap \{|\tilde{Z}_{T_{k,n}}^{(k)}| > \e s_n \} } \right]
\leq \E\left[ \tilde{Z}_{T_{k,n}}^2 \ind_{ \e s_n < |\tilde{Z}_{T_{k,n}}| \leq \sqrt{T_{k,n}}\log^4 T_{k,n} } \right].
\]
Note that this shows that the expectation above is zero unless $\sqrt{T_{k,n}} \log^4 T_{k,n} > \e s_n$, and since for $n$ large enough we have $\e s_n \geq \sqrt{s_n} \log^4 s_n$, we can conclude that the above expectation is zero unless $T_{k,n} > s_n$. 
To bound this truncated second moment in the case $T_k > s_n$ we would like to use Lemma \ref{cor-L2trunc}, 
but to apply this we need to first check that that $\e s_n > \sqrt{T_k \log T_k}$. 
However, 
it follows from Theorem \ref{k2-meanvar} and the definition of  $I_{n,\delta_n}^-$ that 
for $n$ sufficiently large we have
$T_{k,n}\log T_{k,n} \leq C \delta_n s_n^2 \leq \epsilon^2 s_n^2$ for all $k \in I_{n,\delta_n}^-$. Therefore, we can conclude for $n$ large and $k \in I_{n,\delta_n}^-$ that 
\begin{align}
\E\left[ (\tilde{Z}_{T_{k,n}}^{(k)})^2 \ind_{A_{k,n} \cap \{|\tilde{Z}_{T_{k,n}}^{(k)}| > \e s_n \} } \right]
&\leq C \left( T_{k,n} \log\log T_{k,n} +  \frac{T_{k,n}^2 \log^2 T_{k,n}}{\e^2 s_n^2} \right) \ind_{T_{k,n} > s_n} \nonumber \\
&\leq C \sigma_{k,n}^2 \left( \frac{\log\log T_{k,n}}{\log T_{k,n}} + \frac{\sigma_{k,n}^2}{\e s_n^2} \right) \ind_{T_{k,n} > s_n} \nonumber \\
&\leq  C \sigma_{k,n}^2 \left( \frac{\log\log s_n}{\log s_n} + \frac{\delta_n}{\e} \right). \label{Lind2a}
\end{align}
For the second expectation in the sum in \eqref{Lind2}, recall that if $T_{k,n} \leq \sqrt{s_n}$ then $A_{k,n}^c = \emptyset$ and thus 
this expectation is zero. 
On the other hand, since 
\[
|\E[ \tilde{Z}_{T_{k,n}}^{(k)} \ind_{A_{k,n}} ] | = |\E[ \tilde{Z}_{T_{k,n}}^{(k)} \ind_{A_{k,n}^c} ]|  \leq \E[ |\tilde{Z}_{T_{k,n}}| \ind_{|\tilde{Z}_{T_{k,n}} > \sqrt{T_{k,n}}\log^4 T_{k,n}} ],
\]
we can also use Lemma \ref{cor-L1tail} to obtain the bound 
\begin{equation}\label{Lind2b}
\E[ \tilde{Z}_{T_{k,n}}^{(k)} \ind_{A_{k,n}} ]^2
\leq \frac{C T_{k,n}}{\log^8 T_{k,n}} \ind_{T_{k,n} > \sqrt{s_n}}
\leq \frac{C \sigma_{k,n}^2}{\log^9 T_{k,n}} \ind_{T_{k,n} > \sqrt{s_n}}
\leq  \frac{C \sigma_{k,n}^2 }{\log^9 s_n}. 
\end{equation}
Combining the estimates in \eqref{Lind2a} and \eqref{Lind2b} 
and noting $\sum_{k \in I_{n,\delta_n}^-} \sigma_{k,n}^2 \leq \sum_{k=1}^{\ell_n+1} \sigma_{k,n}^2 = s_n^2$, one obtains \eqref{Lind2}. 
This completes the verification of the Lindeberg condition and completes the proof of the lemma. 
\end{proof}

Having completed the preparatory steps, we are now ready to give the proof of Theorem \ref{thm:RWCRElim-k2}. 

\begin{proof}[Proof of Theorem \ref{thm:RWCRElim-k2}]
Let $\delta_n \to 0$ be as in Corollary \ref{cor-step2} and let 
\begin{equation}\label{bnform}
b_n^2 = \frac{\beta^2 s_{n,\delta_n,+}^2 + \tilde{s}_{n,\delta_n}^2}{s_n^2}.
\end{equation}
Note that since
\[
 \Var\left( \tilde{Z}_{T_{k,n}}^{(k)} \ind_{A_{k,n}} \right)
 \leq \E\left[ \left( \tilde{Z}_{T_{k,n}}^{(k)}  \right)^2\ind_{A_{k,n}} \right] 
 \leq \E\left[ \left( \tilde{Z}_{T_{k,n}}^{(k)}  \right)^2 \right] = \sigma_{k,n}^2,
\]
we have 
that $\tilde{s}_{n,\delta_n}^2 \leq s_{n,\delta_n,-}^2$, 
and because $s_n^2 = s_{n,\delta_n,+}^2 + s_{n,\delta_n,-}^2$ and $\beta<1$, it then follows that $b_n \leq 1$ for all $n$. 
Moreover, since we have shown that \eqref{tsndnrat} holds whenever $\liminf_{n\to\infty} \frac{s_{n,\delta_n,-}}{s_n} > 0$, we can also conclude that 
$\liminf_{n\to\infty} b_n \geq \beta$. 
Therefore, if we can prove that 
\begin{equation}\label{Glimbn}
 \frac{X_n - \E[X_n]}{b_n s_n} \underset{n\to\infty}{\Longrightarrow} \Phi, 
\end{equation}
then the conclusion of Theorem \ref{thm:RWCRElim-k2} will hold with $\beta_n = b_n \vee \beta$. 

We will prove  
\eqref{Glimbn}
by proving that every subsequence has a further subsequence that converges to a standard Gaussian. 
To this end, let $n_j \to \infty$ be a fixed subsequence and consider the following cases. 

\textbf{Case 1: $\liminf_{j\to\infty} \frac{s_{n_j,\delta_{n_j},-}}{s_{n_j}} = 0$.}
By passing to a further subsequence we can assume that $\lim_{j\to\infty} \frac{s_{n_j,\delta_{n_j},-}}{s_{n_j}} = 0$. It then follows from Corollary \ref{cor-step3} that $\frac{X_{n_j} - \E[X_{n_j} ]}{\beta s_{n_j}} \Rightarrow \Phi$. 
Since $\tilde{s}_{n,\delta_n} \leq s_{n,\delta_n,-}$ always holds, then the definition of $b_n$ and the assumption on the subsequence in this case implies that $\frac{b_{n_j}s_{n_j}}{\beta s_{n_j}} \to 1$ as $j\to \infty$, so that $\frac{X_{n_j} - \E[X_{n_j} ]}{b_{n_j} s_{n_j}} \Rightarrow \Phi$ as $j\to\infty$.

\textbf{Case 2:} $s_{n_j,\delta_{n_j},+} = 0$ for infinitely many $j\geq 1$.
In this case, by passing to a further subsequence we can assume that $s_{n_j,\delta_{n_j},+} = 0$ for all $j$, in which case 
$X_{n_j} - \E[ X_{n_j} ] = \mathcal{Z}_{n_j,\delta_{n_j}}^-$,  
$s_{n_j,\delta_{n_j},-} = s_n$, and $b_{n_j} s_{n_j} = \tilde{s}_{n_j,\delta_{n_j}}$ so that Lemma \ref{lem-smallparts} implies that 
\[
\frac{X_{n_j} - \E[ X_{n_j} ] }{b_{n_j} s_{n_j} } 
= \frac{\mathcal{Z}_{n_j,\delta_{n_j}}^-}{\tilde{s}_{n_j,\delta_{n_j}}}
\underset{j\to\infty}{\Longrightarrow} \Phi.
\]

\textbf{Case 3:} $\liminf_{j\to\infty} \frac{s_{n_j,\delta_{n_j},-}}{s_{n_j}} > 0$ and $s_{n_j,\delta_{n_j},+} > 0$ for all but finitely many $j$.
Since $s_{n_j,\delta_{n_j},+}>0$ for all $j$ large enough we can decompose 
\begin{equation}\label{gendec}
\frac{X_{n_j} - \E[ X_{n_j} ] }{b_{n_j} s_{n_j} } 
= \frac{\beta s_{n_j,\delta_{n_j},+} }{b_{n_j} s_{n_j}} \frac{\mathcal{Z}_{n_j,\delta_{n_j}}^+}{\beta s_{n_j,\delta_{n_j},+} } + \frac{\tilde{s}_{n_j,\delta_{n_j}}}{b_{n_j} s_{n_j}} \frac{\mathcal{Z}_{n_j,\delta_{n_j}}^-}{\tilde{s}_{n_j,\delta_{n_j}}}. 
\end{equation}
Corollary \ref{cor-step2} implies that $\frac{\mathcal{Z}_{n_j,\delta_{n_j}}^+}{\beta s_{n_j,\delta_{n_j},+} } \Rightarrow \Phi$, and 
Lemma \ref{lem-smallparts}
gives that $\frac{\mathcal{Z}_{n_j,\delta_{n_j}}^-}{\tilde{s}_{n_j,\delta_{n_j}}} \Rightarrow \Phi$ also. Also, the two terms on the right side of \eqref{gendec} are independent random variables and the squares of the coefficients of the two terms sum to 1 by the definition of $b_n$. This implies that the right side of \eqref{gendec} converges to $\Phi$ in distribution as $j\to\infty$. 
\end{proof}

A disadvantage to the above proof of Theorem \ref{thm:RWCRElim-k2} is that the formula for the scaling multiple $\beta_n$ depends on the choice of the sequence $\delta_n \to 0$ in Corollary \ref{cor-step2} which is non-explicit. 
The following lemma gives another sequence that is asymptotically equivalent to the scaling constants used in the proof above, but which has the advantage of not relying on the choice of $\delta_n$ and thus which can be used to compute the scaling constants $\beta_n$ for certain choices of cooling maps. 

\begin{lemma}
 The sequence $\beta_n$ in Theorem \ref{thm:RWCRElim-k2} can be chosen as 
\begin{equation}\label{tbnform}
\beta_n = \tilde{b}_n \vee \beta, 
\quad\text{where}\quad
\tilde{b}_n = \frac{\sum_{k=1}^{\ell_n+1} \Var\left( \tilde{Z}_{T_{k,n}} \mathbf{1}_{A_{k,n} }\right) }{s_n^2}, 
\end{equation}
and where the set $A_{k,n}$ is defined as in \eqref{tsndn}.
\end{lemma}

\begin{proof}
Comparing \eqref{tbnform} with \eqref{bnform} and \eqref{tsndn}, we see that it's enough to prove that 
\begin{equation}\label{tvarsc}
\lim_{n\to\infty} \sup_{k \in I_{n,\delta_n}^+} \left| \frac{\Var\left( \tilde{Z}_{T_{k,n}}^{(k)} \mathbf{1}_{A_{k,n} }\right)}{\beta^2 \sigma_{k,n}^2} - 1 \right| = 0, 
\end{equation}
where again $\delta_n$ is the sequence from Corollary \ref{cor-step2} (if $I_{n,\delta_n}^+ = \emptyset$ then the supremum in the display above is by convention taken to be zero).
Note that we are always free to pick the sequence $\delta_n\to 0$ slow enough so that $\delta_n \geq \frac{1}{\log s_n}$. If this is the case, then for $n$ large enough and $k \in I_{n,\delta_n}^+$ we have that 
\[
 \frac{s_n}{\sqrt{\log s_n}} \leq \sqrt{\delta_n s_n^2} \leq \sigma_{k,n} \leq C \sqrt{T_{k,n} \log T_{k,n} } \leq \sqrt{T_{k,n} } \log^4 T_{k,n}. 
\]
Therefore,  for $n$ large enough and $k \in I_{n,\delta_n}^+$ we have 
\[
\Var\left( \tilde{Z}_{T_{k,n}}^{(k)} \mathbf{1}_{A_{k,n} }\right)
= \Var\left( \tilde{Z}_{T_{k,n}}^{(k)} \mathbf{1}_{ |\tilde{Z}_{T_{k,n}}^{(k)}| \leq \sqrt{T_{k,n}} \log^4 T_{k,n} }\right), 
\quad \text{and}\quad 
T_{k,n} \geq \frac{s_n^2}{\log^5 s_n}.
\] 
From this, it follows that for $n$ large enough we have 
\[
\sup_{k \in I_{n,\delta_n}^+} \left| \frac{\Var\left( \tilde{Z}_{T_{k,n}}^{(k)} \mathbf{1}_{A_{k,n} }\right)}{\beta^2 \sigma_{k,n}^2} - 1 \right|
\leq 
\sup_{m\geq \frac{s_n^2}{\log^5 s_n}} \left| \frac{ \Var\left( \tilde{Z}_m \ind_{|\tilde{Z}_m| \leq \sqrt{m}\log^4 m } \right) }{\beta^2 \Var(Z_m) } - 1 \right|, 
\]
and then \eqref{tvarsc} follows from this together with Lemma \ref{cor-smalltsm}. 
\end{proof}

\section{Examples}\label{sec:ex}
In this section we consider specific cooling maps that display interesting/illustrative behaviour in the study of the limit distribution of RWCRE in the subbalistic ($\kappa \in(0,1)$) and in the Gaussian critical ($\kappa = 2$) regime.
To explore the features of RWCRE, we consider both regular cooling maps (maps for which $T_k = \tau_k - \tau_{k-1}$ admits an asymptotic behavior) as well as some cooling maps with more irregular behavior. 

\subsection{Sub-balistic regime}

For the examples in this subsection we will assume that the distribution $\mu$ on environments is $\kappa$-regular with $\kappa \in (0,1)$. 

\begin{example}[Polynomial cooling when $\kappa \in (0,1)$]\label{ex:k1poly}
Let $T_k \sim A k^\alpha$ for some constants $A,\alpha>0$. 
Since $T_k\to\infty$ we can use Corollary \ref{Cor:lambda_explicit} to determine the limiting distributions. It is easy to check for this example that $\max_k \tilde{\lambda}_{\tau,n}(k) \leq \frac{C}{\sqrt{\ell_n}} \to  0$, 
and thus we can conclude that
$\frac{X_n-\E[X_n]}{\sqrt{\Var(X_n)}} \Rightarrow \Phi$. 
Moreover, we can use Theorem \ref{thm:lp_cov} to replace the scaling by the standard deviation  of $X_n$ with a more explicit scaling in this case. Indeed, using $\Var(Z_n) \sim \sMit n^{2\kappa}$ we have in this case that
\[
\Var(X_n) = \sum_{k=1}^{\ell_n} \Var(Z^{(k)}_{T_k}) + \Var(Z^{(\ell_n +1)}_{n-\tau(\ell_n)} ) \sim \frac{\sMit^2 A^{2\kappa}}{2\alpha \kappa+1} \ell_n^{2\alpha \kappa+1},
\qquad \text{as } n\to\infty.
\]
Since $\ell_n \sim \left( \frac{\alpha+1}{A} \right)^{\frac{1}{\alpha+1}} n^{\frac{1}{\alpha+1}}$ we can then conclude that
\[
\frac{X_n-\E[X_n]}{\sigma_{A,\alpha} n^{\frac{2\alpha\kappa+1}{2(\alpha+1)}} }
\underset{n\to\infty}{\Longrightarrow} \Phi, 
\quad \text{where } 
\sigma_{A,\alpha}^2 = \frac{\sMit^2 A^{2\kappa} }{2\alpha\kappa+1} \left( \frac{\alpha+1}{A} \right)^{\frac{2\alpha\kappa+1}{\alpha+1}}. 
\]
Note that the scaling exponent $\frac{2\alpha\kappa+1}{2(\alpha+1)}$ converges to $\frac{1}{2}$ as $\alpha\to 0$ and to $\kappa$ as $\alpha \to \infty$ (if $\kappa=1/2$ then the scaling exponent is always to $1/2$ for all $\alpha>0$). 
\end{example}

\begin{example}[Exponential cooling when $\kappa \in (0,1)$]\label{ex:k1exp}
Let $T_k \sim Ce^{ck}$ for $C,c>0$. 
Again we can use Corollary \ref{Cor:lambda_explicit} to determine the limiting distributions.
In this case one can only obtain limiting distributions along certain subsequences, but the limiting distribution is always a sum of independent (normalized) Mittag-Leffler random variables. For simplicity we will only describe the limiting distribution along the subsequence $n_j = \tau(j)$.
For this choice of $n_j$ we have 
\[
(V_{n_j})^2 = \sum_{k=1}^j (T_k)^{2\kappa} 
\sim \sum_{k=1}^j (C e^{ck} )^{2\kappa} 
\sim \frac{(C e^c)^{2\kappa} }{e^{2c\kappa}-1} e^{2c\kappa j}, \quad \text{as } j\to \infty. 
\]
Since for $k\geq 1$ fixed and $j$ large enough we have that $\tilde{\lambda}^{\downarrow}_{\tau,n_j}(k) = \frac{(T_{j-k+1})^\kappa}{V_{n_j}}$, it follows that 
\[
\lim_{j\to\infty} \lambda^{\downarrow}_{\tau,n_j}(k) = \lambda_{c,*}(k):= \sqrt{(\theta_c^{-2}-1)}\, (\theta_c)^k, \quad \forall k\geq 1, \qquad \text{where } \theta_c = e^{-c\kappa}. 
\]
Since $\sum_{k\geq 1} (\lambda_{c,*}(k))^2 = 1$ we can conclude from Corollary \ref{Cor:lambda_explicit} that 
\begin{equation}\label{sslim-exp}
 \frac{X_{n_j} - \E[X_{n_j}]}{\sqrt{\Var(X_{n_j})}} \underset{j\to\infty}{\Longrightarrow} \left( \cMit \right)^{\otimes \lambda_{c,*}}. 
\end{equation}
We can also obtain more explicit centering and scaling terms for this example. Indeed, it follows from Theorem \ref{thm:lp_cov} that 
$\E[X_{n_j}] \sim \frac{\mu_{\mathfrak{M}} (C e^c)^\kappa}{e^{c\kappa}-1} e^{c\kappa j}$ and 
$\sqrt{\Var(X_{n_j})} \sim \frac{\sMit (C e^c)^\kappa }{\sqrt{e^{2c\kappa}-1}} e^{c\kappa j}$,
and since $n_j = \tau(j) \sim \frac{C e^c}{e^c-1} e^{cj}$ implies that $(C e^c)^\kappa e^{c\kappa j} \sim (e^c-1)^\kappa n_j^\kappa$ we can re-write these asymptotics as 
$\E[X_{n_j}] \sim \frac{\mu_{\mathfrak{M}} (e^c-1)^\kappa}{e^{c\kappa}-1} n_j^\kappa$ and 
$\sqrt{\Var(X_{n_j})} \sim \frac{\sMit (e^c-1)^\kappa }{\sqrt{e^{2c\kappa}-1}} n_j^\kappa$.
Since $\sum_{k\geq 1} \lambda_{c,*}(k) = \frac{\sqrt{e^{2c\kappa}-1}}{e^{c\kappa}-1}$, we have by Remark \ref{rem:replacemean} that we can remove the centering terms from both the left and right sides of \eqref{sslim-exp}. Finally, since $\frac{\sqrt{\Var(X_{n_j})}}{n_j^\kappa} \to \frac{\sMit(e^c-1)^\kappa}{\sqrt{e^{2c\kappa}-1}}$ 
we obtain the simplified form of the limiting distribution where the scaling is the same as in the RWRE case
\[
\frac{X_{n_j} }{ n_j^\kappa } 
\underset{j\to\infty}{\Longrightarrow}
\frac{(e^c-1)^\kappa }{\sqrt{e^{2c\kappa}-1}} \sum_{k=1}^\infty \lambda_{c,*}(k) \Mit^{(k)}
= (e^c-1)^\kappa \sum_{k=1}^\infty e^{-c \kappa k} \Mit^{(k)}.
\]

\end{example}

\begin{example}[super exponential cooling when $\kappa \in (0,1)$]
If $\log(T_k) \sim e^{ck}$ for some $c>0$, 
then all subsequential limits are a linear combination of one or two independent Mittag-Leffler random variables.
To give a specific example consider the cooling map with $\tau(j) = 2^{2^j}$ for $j\geq 1$. Then, fix a parameter $\theta \geq 0$ and consider the subsequence $n_j = \fl{(1+\theta) 2^{2^j}}$. Note that for $j$ large enough we have $\tau(j) \leq n_j < \tau(j+1)$, so that the representation in \eqref{basicdec} becomes $X_{n_j} \overset{\text{Law}}{=} \sum_{k=1}^j Z^{(k)}_{T_j} + Z^{(j+1)}_{\fl{\theta 2^{2^j}}}$. 
One can show
that only the last two terms in the sum on the right survive in the limiting distribution. 
Indeed, since the variance asymptotics in Theorem \ref{thm:lp_cov} imply that $\Var(X_n) \sim \Var(Z_{T_j}) + \Var(Z_{\fl{\theta 2^{2^j}}}) \sim \sMit^2 (1+\theta^{2\kappa}) n_j^{2\kappa}$, we can then apply Corollary \ref{Cor:lambda_explicit} to get 
\[
 \frac{X_{n_j}-\E[X_{n_j}] }{\sqrt{\Var(X_{n_j})}} \underset{j\to\infty}{\Longrightarrow} \frac{1}{\sqrt{1+\theta^{2\kappa}}} \left( \frac{ \Mit^{(1)}-\mu_{\mathfrak{M}}}{\sMit} \right) + \frac{\theta^\kappa }{\sqrt{1+\theta^{2\kappa}}} \left( \frac{ \Mit^{(2)}-\mu_{\mathfrak{M}}}{\sMit} \right). 
\]
Finally, since the mean and variance asymptotics in Theorem \ref{thm:lp_cov} imply that $\E[X_{n_j}] \sim \mu_{\mathfrak{M}}(1+\theta^\kappa)n_j^\kappa$ and 
$\sqrt{\Var(X_{n_j})} \sim \sMit \sqrt{1+\theta^{2\kappa}} n_j^\kappa$ we can rewrite the limiting distribution above as 
\[
 \frac{X_{n_j}}{n_j^\kappa} \underset{j\to\infty}{\Longrightarrow} 
 \Mit^{(1)} + \theta^\kappa \Mit^{(2)}. 
\]
\end{example}

\begin{example}[Mixtures of Mittag-Leffler and Gaussian when $\kappa \in (0,1)$]\label{ex:k1exp-mix}
The basic idea to build mixtures of Mittag-Leffler random variables with Gaussian  is to build a cooling map by interweaving a fast growing cooling map where the limiting distribution is a sum of Mittag-Leffler distributions with a slow growing cooling map where the limiting distribution is Gaussian. 
To give a specific example of this, let $\tau$ be the cooling map with cooling intervals given by 
\[
T_{2^i} = \left\lfloor 2^{(i-1)/(2\kappa)} \right\rfloor \quad \text{for } i\geq 1, \quad \text{and}\quad T_k = 1 \quad \text{if } k \notin \{2^i: \ i\in \N\}. 
\]
We will compute the limiting distribution along the subsequence $n_j = \tau(2^j)$. 
To this end, first note that from Theorem \ref{thm:lp_cov} that 
\begin{align*}
 \Var(X_{n_j}) &= \sum_{i=1}^j \Var\left(Z_{\fl{2^{(i-1)/(2\kappa)}}} \right) + (2^j-j) \Var(Z_1) \\
 &\sim \sum_{i=1}^j \sMit^2 2^{i-1} + 2^j \Var(Z_1) 
 \sim \left(\sMit^2 + \Var(Z_1) \right) 2^j, \quad \text{as } j\to \infty. 
 \end{align*}
Moreover, since the $k$-th largest cooling interval among the first $2^j$ cooling intervals is $T_{2^{j-k+1}} = \fl{2^{(j-k)/(2\kappa)}}$ we get that for any fixed $k\geq 1$,
\[
\lim_{j\to\infty} \lambda^\downarrow_{\tau,n_j}(k) 
= \lim_{j\to\infty} \sqrt{ \frac{ \Var\left(Z_{\fl{2^{(j-k)/(2\kappa)}}}\right) }{\Var(X_{n_j})} } 
= \sqrt{ \frac{\sMit^2}{\sMit^2 + \Var(Z_1)} } 2^{-k/2}
=: \lambda_*(k). 
\]
Since $a(\lambda_*) = \left( \frac{\Var(Z_1)}{\sMit^2 + \Var(Z_1)} \right)^{1/2} \in (0,1)$,
we get in this case that the limiting distribution is a mixture of sums of independent Mittag-Leffler random variables and an independent Gaussian. 
More precisely, 
\[
\frac{X_{n_j}-\E[X_{n_j}]}{\sqrt{\Var(X_{n_j})}}  
 \underset{j\to\infty}{\Longrightarrow}
\left( \cMit \right)^{\otimes \lambda_*}
+  \left( \frac{\Var(Z_1)}{\sMit^2 + \Var(Z_1)} \right)^{1/2} \Phi. 
\]
\end{example}

\begin{example}[Arbitrary mixtures of Mittag-Leffler and Gaussian when $\kappa \in (0,1)$]\label{ex:arbMLGmix}
 A natural question is whether or not given some $\lambda_* \in \ell^2$ with $\sum_{k\geq 1} \lambda_*(k)^2 \leq 1$ one can find a cooling map and a subsequence $n_j\to \infty$ such that one has a limiting distribution of the form $\left( \cMit \right)^{\otimes \lambda_*}
+ a(\lambda_*)\Phi$ as in \eqref{Lawsub}. In this example we give algorithms showing how this can indeed be done. 

Without loss of generality, we may always assume that $\lambda_*(k)$ is non-increasing in $k$. Our algorithm will be slightly different depending on whether or not $\lambda_*(k)$ is eventually zero. 
In both cases, however, we will use an iterative method to construct a cooling map $\tau$ along with a sequence $\{N_j\}_{j\geq 1}$, and then we will let $n_j = \tau(N_j)$ for $j\geq 1$. 
The cooling map will have the property that $\lim_{k\to\infty} T_k = \infty$ so that we may apply Corollary \ref{Cor:lambda_explicit} to identify the subsequential limiting distribution. 

\noindent\textbf{Case I: $\lambda_*(k) > 0$ for all $k\geq 1$.}
We begin by defining the sequence $\{N_j\}_{j\geq 1}$ by letting $N_0=0$ and then letting 
\[
 N_j = N_{j-1} + j + K_j, \qquad \text{where } \quad K_j = \left\lfloor \left( \frac{a(\lambda_*)}{\lambda_*(j)} \right)^2 \right\rfloor, \qquad j\geq 1.  
\]
Before constructing the cooling map, we recall the  
scaling factor $V_n$ defined in Corollary \ref{Cor:lambda_explicit} and let $\mathfrak{V}_0=1$ and
\begin{equation}\label{mfVjdef}
 \mathfrak{V}_j = V_{\tau(N_j)} = \sqrt{ \sum_{k\leq N_j} (T_k)^{2\kappa} }, \qquad j\geq 1. 
\end{equation}
We will now define the cooling intervals $\{T_k\}_{k\geq 1}$ as follows.
Since each integer $k\geq 1$ is in a unique interval $(N_{j-1},N_j]$ for some $j\geq 1$ we let 
\[
 T_k = 
 \begin{cases}
  \left\lceil \left(\frac{\mathfrak{V}_{j-1} \lambda_*(\ell)}{\lambda_*(j)} \right)^{1/\kappa} \right\rceil 
  &\text{if } k=N_{j-1}+\ell \text{ for some } \ell = 1,2,\ldots j \\
  \left\lceil \mathfrak{V}_{j-1}^{1/\kappa} \right\rceil & \text{if } N_{j-1} + j < k \leq N_j. 
 \end{cases}
\]

Now, it follows easily from \eqref{mfVjdef} and our assumption that $\lambda_*$ is non-increasing that for $j\geq \ell\geq 1$ we have 
$ \max_{k\leq N_{j-1}} T_k \leq \mathfrak{V}_{j-1}^{1/\kappa} \leq T_{N_{j-1} + j} \leq T_{N_{j-1}+\ell}$, 
and thus (recalling the definition of $\tilde{\lambda}_{\tau,n}$ in Corollary \ref{Cor:lambda_explicit})  we can conclude that 
$\tilde{\lambda}_{\tau,n_j}^\downarrow(\ell) = \tilde{\lambda}_{\tau,n_j}(N_{j-1}+\ell) = \frac{ (T_{N_{j-1}+\ell})^{\kappa} }{\mathfrak{V}_j}$
for $j\geq \ell$. Using this we can then conclude that for $\ell\geq 1$ fixed we have that 
\begin{align*}
 \lim_{j\to\infty} \widetilde{\lambda}_{\tau,n_j}^\downarrow (\ell)
 &= \lim_{j\to\infty} \frac{ \left\lceil \left(\frac{\mathfrak{V}_{j-1} \lambda_*(\ell)}{\lambda_*(j)} \right)^{1/\kappa} \right\rceil^{\kappa} }{\sqrt{\mathfrak{V}_{j-1}^2 + \sum_{i\leq j} \left\lceil \left(\frac{\mathfrak{V}_{j-1} \lambda_*(i)}{\lambda_*(j)} \right)^{1/\kappa} \right\rceil^{2\kappa} + K_j \left\lceil \mathfrak{V}_{j-1}^{1/\kappa} \right\rceil^{2\kappa} }} \\
 &= \lim_{j\to \infty} \frac{\lambda_*(\ell)}{\sqrt{ \lambda_*(j)^2 + \sum_{i\leq j} \lambda_*(i)^2 + a(\lambda_*)^2}} \\
 &= \frac{\lambda_*(\ell)}{\sum_{i\geq 1} \lambda_*(i)^2 + a(\lambda_*)^2} 
 = \lambda_*(\ell),  
\end{align*}
where in the last equality we used that $a(\lambda_*)^2 = 1 - \sum_{i\geq 1} \lambda_*(i)^2$. 

\noindent\textbf{Case II: there is a $k_0 \geq 1$ such that $\lambda_*(k) > 0 \iff k \leq k_0$.}
The algorithm is very similar in this case, but with the following changes. Now we define the sequence $\{N_j\}_{j\geq 1}$ by 
\[
 N_j = N_{j-1} + k_0 + K_j, \quad \text{where } \quad 
 K_j = \left\lfloor \left( j \, a(\lambda_*) \right)^2 \right\rfloor, \quad j\geq 1, 
\]
and for $k \in (N_{j-1},N_j]$ we let the cooling interval 
\[
 T_k = 
 \begin{cases}
  \left\lceil \left( j \mathfrak{V}_{j-1} \lambda_*(\ell) \right)^{1/\kappa} \right\rceil 
  &\text{if } k=N_{j-1}+\ell \text{ for some } \ell = 1,2,\ldots k_0 \\
  \left\lceil \mathfrak{V}_{j-1}^{1/\kappa} \right\rceil & \text{if } N_{j-1} + k_0 < k \leq N_j. 
 \end{cases}
\]

Next, since for  $j\geq 1/\lambda_*(k_0)$ we have 
\begin{equation} \label{earlyTk}
 \max_{k\leq N_{j-1}} T_k
 \leq \mathfrak{V}_{j-1}^{1/\kappa} 
 \leq \left( j \mathfrak{V}_{j-1} \lambda_*(k_0) \right)^{1/\kappa}
 \leq T_{N_{j-1} + k_0}
 \leq T_{N_{j-1}+\ell},\qquad \text{for } 1\leq \ell \leq k_0,  
 \end{equation}
it follows that for any $1 \leq \ell \leq k_0$ we have $\widetilde{\lambda}_{\tau,n_j}^\downarrow(\ell) = \widetilde{\lambda}_{\tau,n_j}(N_{j-1}+\ell)$ for all $j$ sufficiently large and thus
\begin{align*}
 \lim_{j\to\infty} \widetilde{\lambda}_{\tau,n_j}^\downarrow (\ell)
 &= \lim_{j\to\infty} \frac{\left\lceil \left( j \mathfrak{V}_{j-1} \lambda_*(\ell) \right)^{1/\kappa} \right\rceil^{\kappa} }{ \sqrt{ \mathfrak{V}_{j-1}^2 + \sum_{i\leq k_0} \left\lceil \left( j \mathfrak{V}_{j-1} \lambda_*(i) \right)^{1/\kappa} \right\rceil^{2\kappa}   + \left\lfloor \left( j a(\lambda_*) \right)^2 \right\rfloor \left\lceil \mathfrak{V}_{j-1}^{1/\kappa} \right\rceil^{2\kappa} } } \\
 &= \lim_{j\to\infty} \frac{\lambda_*(\ell)}{ \sqrt{ \frac{1}{j^2} + \sum_{i\leq k_0} \lambda_*(i)^2 + a(\lambda_*)^2}} \\
 &= \lambda_*(\ell). 
\end{align*}

Finally, we need to show that $\lim_{j\to\infty} \widetilde{\lambda}_{\tau,n_j}^\downarrow (k) = 0$ for all $k > k_0$. For this it is of course enough to show that $\lim_{j\to\infty} \widetilde{\lambda}_{\tau,n_j}^\downarrow (k_0 + 1) = 0$. 
To this end, it follows from \eqref{earlyTk} that
\[
 \widetilde{\lambda}_{\tau,n_j}^\downarrow (k_0 + 1)
 \leq \frac{\max_{k\leq N_{j-1}} (T_k)^\kappa}{\mathfrak{V}_j} \leq \frac{\mathfrak{V}_{j-1}}{\mathfrak{V}_j}, 
\]
and since it is easy to see that $\mathfrak{V}_j \sim j \mathfrak{V}_{j-1}$ as $j\to\infty$ then $\lim_{j\to\infty} \widetilde{\lambda}_{\tau,n_j}^\downarrow (k_0 + 1) = 0$. 
\end{example}

\subsection{Gaussian critical regime}

For the examples in this subsection we will assume that the distribution $\mu$ on environments is $2$-regular. The examples below demonstrate the various properties of the sequence of scaling constants $\beta_n$ in Theorem \ref{thm:RWCRElim-k2} that can be obtained by changing the cooling map $\tau$. 
Recall that the scaling constants $\beta_n$ can be given by the formula in \eqref{tbnform}. However, for this formula to be of practical use one needs some way of approximating the truncated variance terms involved. To this end, one can use the following result which follows from the tail bounds for $\tilde{Z}_n$ in Section \ref{sec:cente}. 

\begin{corollary}\label{cor-tvarlim}
Let $(Z_n)_{n\geq 0}$ be a RWRE with a $2$-regular distribution $\mu$ on environments. Then, 
\[
\lim_{n\to\infty} \sup_{x\geq \sqrt{n}\log^4 n} \left| \frac{\Var(\tilde{Z}_n \ind_{|\tilde{Z}_n| \leq x} )}{b^2 n\log n + 2K_0 v n \log\left( \frac{x \wedge (nv/2)}{\sqrt{n}} \right) }  - 1 \right| = 0. 
\]
\end{corollary}

The proof of Corollary \ref{cor-tvarlim} is straightforward (using similar methods as in the proofs of Lemmas \ref{cor-L2trunc} and \ref{cor-smalltsm}), but is rather tedious. 
Since Corollary \ref{cor-tvarlim} is needed only for the justifying the computations of $\beta_n$ in the examples below, we 
give its proof in Appendix \ref{app:tvar}.

The first two examples give families of cooling maps which show that 
we cannot change the condition in Theorem \ref{thm:RWCRElim-k2} that $\beta_n \in [\beta,1]$ to a smaller interval. 

\begin{example}[Polynomial cooling when $\kappa=2$]
If $T_k \sim A k^\alpha$ for some constants $A,\alpha>0$, we claim that the scaling constants $\beta_n$ can be chosen to be equal to a constant $\sigma_\alpha \in (\beta,1]$ which depends only on $\alpha>0$. 
More precisely, 
\begin{equation}\label{polybn}
\frac{X_n - \E[X_n]}{\sigma_\alpha \sqrt{\Var(X_n)} } \underset{n\to\infty}{\Longrightarrow} \Phi, 
\qquad 
\text{where } 
\sigma_\alpha^2 = 
 \begin{cases}
  1 & \text{if } \alpha \leq 1 \\
  \frac{b^2+\frac{K_0 v}{\alpha}}{b^2 + K_0 v} & \text{if } \alpha > 1. 
 \end{cases}
\end{equation}
(Note that $\sigma_\alpha \in (\beta,1]$ for all $\alpha>0$ with $\sigma_\alpha \to \beta$ as $\alpha \to \infty$.)
To simplify computations, we first determine the scaling constants along the subsequence $\tau(n)$ of cooling times. 
It follows from the variance asymptotics in Corollary \ref{k2-meanvar} that 
\begin{equation}\label{stn2}
s_{\tau(n)}^2 = \Var(X_{\tau(n)}) 
\sim (b^2 + K_0 v)A\alpha \sum_{k=1}^n k^\alpha \log k
\sim \frac{(b^2 + K_0 v) A \alpha}{\alpha+1} n^{\alpha+1} \log n. 
\end{equation}
A consequence of this is that $\max_{k\leq n} \sqrt{T_k} \log^4 T_k \leq \frac{s_{\tau(n)}}{\sqrt{\log s_{\tau(n)}}}$ for $n$ sufficiently large, and therefore using \eqref{tbnform} we have that 
\[
\beta_{\tau(n)}^2 \sim \frac{ \sum_{k=1}^n \Var\left( \tilde{Z}_{T_k} \ind_{|\tilde{Z}_{T_k}| \leq \frac{s_n}{\sqrt{\log s_n}}} \right) }{ \frac{(b^2 + K_0 v) A \alpha}{\alpha+1} n^{\alpha+1} \log n }. 
\]
Since it follows from \eqref{stn2} that $\frac{s_{\tau(n)}}{\sqrt{\log s_{\tau(n)}}} \sim C n^{\frac{\alpha+1}{2}}$ for some $C$, we can then use Corollary \ref{cor-tvarlim} to deduce that
\begin{align*}
 \beta_{\tau(n)}^2 
& \sim \frac{ \sum_{k=1}^n \left\{ b^2 T_k \log T_k + 2K_0 v T_k \log\left( \frac{ n^{(\alpha+1)/2}  \wedge T_k }{\sqrt{T_k}} \right) \right\} }{ \frac{(b^2 + K_0 v) A \alpha}{\alpha+1} n^{\alpha+1} \log n } \\
&\sim \frac{ \sum_{k=1}^n \left\{ b^2 \alpha k^\alpha \log k + K_0 v  k^\alpha \log\left( \frac{ n^{\alpha+1}  \wedge k^{2\alpha} }{k^{\alpha}} \right) \right\} }{ \frac{(b^2 + K_0 v)  \alpha}{\alpha+1} n^{\alpha+1} \log n }
\underset{n\to\infty}{\longrightarrow} 
\begin{cases}
1 & \text{if } \alpha \leq 1 \\
\frac{b^2+\frac{K_0 v}{\alpha}}{b^2 + K_0 v} & \text{if } \alpha > 1. 
\end{cases} 
\end{align*}
(Note that in the second line we are replacing $T^k$ by $k^\alpha$ instead of $A k^\alpha$ inside the logarithm of the second term since the multiplicative constant $A$ inside the logarithm doesn't change the asymptotics.)

Thus, we have justified the formula for the scaling constant $\sigma_\alpha$ in \eqref{polybn}, but only along the subsequence $\tau(n)$. 
To justify the general limiting distribution, we decompose 
\[
\frac{X_n-\E[X_n]}{\sigma_\alpha \sqrt{\Var(X_n)} }
= \frac{X_{\tau(\ell_n)} - \E[X_{\tau(\ell_n)}]}{\sigma_\alpha \sqrt{\Var(X_n)}} 
+ \frac{Z_{n-\tau(\ell_n)}^{(\ell_n+1)} - \E[Z_{n-\tau(\ell_n)}] }{\sigma_\alpha \sqrt{\Var(X_n)}}. 
\]
One can then check that $\Var(X_n) \sim \Var(X_{\tau(\ell_n)})$ and that $\Var(Z_{n-\tau(\ell_n)}) = o( \Var(X_n))$ so that the first term on the right converges to a standard Gaussian and the second term on the right converges to zero in distribution. 
Finally, note that for this example one can use Corollary \ref{k2-meanvar} to show that $\Var(X_n) \sim \frac{(b^2+K_0 v)\alpha}{\alpha+1} n\log n$, so that we can write $\frac{X_n - \E[X_n]}{c_\alpha \sqrt{n\log n}} \underset{n\to\infty}{\Longrightarrow} \Phi$ where $c_\alpha = \sigma_\alpha \sqrt{\frac{(b^2+K_0 v)\alpha}{\alpha+1}}$. However, since \eqref{errorratio} does not hold for polynomial cooling we cannot 
use Corollary \ref{k2-meanvar} to justify replacing the centering term $\E[X_n]$ by $nv$.

\end{example}

\begin{example}[Exponential cooling when $\kappa=2$]\label{ex:k2exp}
 Let $\tau$ be a cooling map with exponentially growing cooling intervals $T_k \sim e^{rk}$ for some $r>0$. 
 For this example the cooling intervals grow fast enough that in the decomposition of the variance 
 $\Var(X_n) = \sum_{k=1}^{\ell_n+1} \Var(Z_{T_{k,n}})$, 
 only the ``large" terms in the sum contribute to the asymptotics of the variance. 
 More precisely, using the notation from Section \ref{sec:kappa2} we have that $\lim_{\delta\to 0}\liminf_{n\to\infty} \frac{s_{n,\delta,+}^2}{s_n^2} = 1$, 
 or equivalently $\lim_{\delta\to 0}\limsup_{n\to\infty} \frac{s_{n,\delta,-}^2}{s_n^2} = 0$. 
 Then it follows from Corollary \ref{cor-step3} that for this example we have $\frac{X_n - \E[X_n]}{\beta \sqrt{\Var(X_n)} } \underset{n\to\infty}{\Longrightarrow} \Phi$. 
 Note that we can simplify the limiting distribution in this example using the asymptotics of the mean and variance from Corollary \ref{k2-meanvar}.  
 Indeed, for exponential cooling it can easily be checked that \eqref{errorratio} holds and also that 
 $\sum_k T_{k,n}\log(T_{k,n}) \sim n\log n$, so that we can conclude that 
 \[
  \frac{X_n - nv}{b\sqrt{n\log n}} \underset{n\to\infty}{\Longrightarrow} \Phi, 
 \]
 for this example. 
\end{example}

While the above two examples show that one cannot restrict the scaling constants to an interval smaller than $[\beta,1]$, these examples are all regular enough so that the scaling constant can be a fixed constant and doesn't need to oscillate with $n$. The following gives an explicit example of a cooling map where one cannot obtain a limiting distribution without letting $\beta_n$ depend on $n$. 
One can give somewhat simpler examples which demonstrate this oscillation of $\beta_n$, but the example below has $\liminf_{n\to\infty}\beta_n = \beta$ and $\limsup_{n\to\infty}\beta_n = 1$.

\begin{example}[Full oscillation of multiplicative scaling constant $\beta_n$] 
For $i\geq 1$ let $m_i = 2^{2^i}$ and $r_i = i m_i = i 2^{2^i}$. Then, let
$\tau$ be the cooling map given by 
\[
T_{r_i} = m_i, \quad i\geq 1, 
\quad\text{and}\quad 
T_k = 1 \quad \text{if } k \notin \{ r_i: i \geq 1 \}. 
\]
For a fixed $t \geq 0$ we will consider the distribution of the RWCRE along the subsequence $n_j=n_j(t) = \tau\left( r_j + \fl{t m_j \log(m_j) } \right)$. 
Note that for $j$ large enough we have $\tau(r_j) \leq n_j(t) < \tau(r_{j+1})$, and thus for $j$ large enough we can decompose $X_{n_j(t)} - \E[X_{n_j(t)}]$ as 
\begin{equation}\label{oscex-dec}
X_{n_j(t)} - \E[X_{n_j(t)}] = 
\sum_{i=1}^j \left( Z_{m_i}^{(r_i)} - \E[Z_{m_i}] \right) 
+ \sum_{\substack{ k\leq r_j + \fl{t m_j \log(m_j)}  \\ k\notin \{r_i: i\geq 1 \} }} (Z_1^{(k)} - \E[Z_1]).
\end{equation}
As $j\to \infty$, the variance of the first term on the right in \eqref{oscex-dec} is asymptotic to 
\[
(b^2 + K_0 v)\sum_{i=1}^j m_i \log(m_i) 
\sim (b^2 + K_0 v) m_j \log(m_j),
\]
while the variance of the second term on the right is 
\[
(r_j + \fl{t m_j \log(m_j)} - j) \Var(Z_1) \sim 
\begin{cases}
 r_j \Var(Z_1) & \text{if } t=0 \\
 t m_j \log(m_j) \Var(Z_1) & \text{if } t>0. 
\end{cases}
\]
Since the second term in \eqref{oscex-dec} is a sum of i.i.d.\ random variables, it converges in distribution to a standard Gaussian when scaled by its standard deviation, while for the first term in \eqref{oscex-dec} we can apply Theorem \ref{thm:RWCRElim-k2} to the cooling map $\tau'$ with increments $T'_{k} = m_k = 2^{2^k}$ to get that this sum converges to a standard Gaussian when scaled by $\beta$ times its standard deviation. 
From this we see that 
\[
\frac{ X_{n_j(t)} - \E[X_{n_j(t)}] }{\alpha_t \sqrt{\Var(X_{n_j(t)})} } \underset{j\to\infty}{\Longrightarrow} \Phi, 
\quad
\text{where } 
\alpha_t^2 
= \frac{b^2 + t\Var(Z_1)}{b^2 + K_0 v + t \Var(Z_1)}.
\]
Finally, note that $\alpha_0 = \beta$ and $\alpha_t \nearrow 1$ as $t\nearrow \infty$. This shows that in applying Theorem \ref{thm:RWCRElim-k2} to the cooling sequence $\tau$ in this example, not only does one need to let the scaling constant $\beta_n$ vary with $n$, but also that the sequence $\beta_n$ will continue to oscillate through the entire interval $[\beta,1]$. 
\end{example}

\appendix
\section{Technical results for the tail estimates}

In this appendix we collect some results for sums of i.i.d.\ random variables that are needed in Section \ref{sec:kappa2}.
Some of these results may be already known, but we include them here for completeness.

\begin{lemma}\label{lem:stable-log-factor}
Assume $\xi_1$ has mean zero, is bounded below by $-L$ for some $L>0$, and has right tail decay $P(\xi_1>x)=O(x^{-2})$. Then, there exists a constant $C>0$ that depends on the distribution of $\xi_1$ such that
\[
    E[e^{-\lambda\xi_1}]\leq e^{C\lambda^2|\log\lambda|}, \quad \text{for all } \lambda \in (0,1/e). 
\]
\end{lemma}
\begin{proof}
Define $\hat{\xi}_1=\xi_1+L$ so that $\hat{\xi}_1$ is non-negative and $E[\hat{\xi}_1]=L$. 
We begin by noting that 
\begin{align*}
 e^{-\lambda L}E[e^{-\lambda\xi_1}]
    =E[e^{-\lambda\hat{\xi}_1}]
    &= 1 - \lambda L + E[e^{-\lambda\hat{\xi}_1}-1+\lambda \hat{\xi}_1]\\
    &= 1 - \lambda L + \lambda \int_0^\infty (1-e^{-\lambda x}) P( \hat{\xi}_1 > x ) \, \ud x \\
    &\leq 1 - \lambda L + \lambda \int_0^\infty \min\{ \lambda x, 1\} P( \hat{\xi}_1 > x ) \, \ud x. 
\end{align*}
Since $\lambda <1$,
bounding the probability inside the integral by $1$ when $x<1$ and by $Kx^{-2}$ when $x\geq 1$ we obtain 
\begin{align*}
 e^{-\lambda L}E[e^{-\lambda\xi_1}]
&\leq 1-\lambda L+\lambda^2\int_0^1 x \, \ud x +K\lambda^2\int_1^{\lambda^{-1}} x^{-1} \, \ud x +K\lambda \int_{\lambda^{-1}}^\infty x^{-2} \, \ud x\\
    &= 1-\lambda L+\frac{1}{2}\lambda^2+K\lambda^2|\log \lambda|+K\lambda^2\\
    &\leq \exp\left\{-\lambda L+\left(2K+\frac{1}{2}\right)\lambda^2|\log\lambda|\right\}, 
\end{align*}
where in the last inequality we also used that $\lambda \in (0,1/e)$ implies that $|\log \lambda| > 1$. 
Finally, the proof is completed by multiplying both sides of the above inequality by $e^{\lambda L}$. 
\end{proof}

\begin{corollary}\label{cor:lthtsum}
Let $\xi_1,\xi_2,\ldots$ be i.i.d.\ random variables which are bounded below and have right tail decay $P(\xi_1 > x) = O(x^{-2})$. Then, for any $a>0$ there exist constants $c,C'>0$ (depending on $a$) such that 
\[
 P\left( \sum_{i=1}^n \xi_i \leq - x \right) \leq C' e^{-c \frac{x^2}{n\log n}}, 
 \qquad \forall x \in (0,an].  
\]
\end{corollary}

\begin{proof}
First of all, for any $\lambda \in (0,1/e)$ we have from Lemma \ref{lem:stable-log-factor} that 
\[
P\left( \sum_{i=1}^n \xi_i \leq - x \right) \leq e^{-\lambda x} \left( E\left[ e^{-\lambda \xi_1}\right] \right)^n \leq \exp\left\{ -\lambda x + Cn  \lambda^2 |\log \lambda| \right\}. 
\]
It is not  easy to find a value of $\lambda$ that minimizes the upper bound on the right, but we can achieve a nearly optimal upper bound by choosing $\lambda$ such that $\lambda |\log \lambda| = \frac{x}{2Cn}.$
Since the function $\lambda \mapsto \lambda |\log \lambda|$ achieves its maximum of $1/e$ at $\lambda = 1/e$, we can find such a $\lambda \in (0,1/e)$ only if $x \leq \frac{2C}{e} n$. Thus, for now we will restrict ourselves to $x < \frac{2C}{e} n$ and will extend our bound to $x \leq an$ later.  
With this choice of $\lambda$ we then have 
\begin{equation}
\label{Snltbound1}
P\left( \sum_{i=1}^n \xi_i \leq - x \right) 
\leq \exp\left\{ \frac{-\lambda|\log\lambda| x + C n \lambda^2 |\log \lambda|^2}{|\log \lambda|} \right\}
= \exp\left\{ \frac{-x^2}{4Cn |\log \lambda|} \right\}.  
\end{equation}
Next, we claim that our choice of $\lambda$ above implies that 
\begin{equation}\label{loglub}
 |\log \lambda| \leq 2 \log\left( \frac{2Cn}{x} \right) . 
\end{equation}
To see this, note that $t^2 |\log(t^2)| = 2t (t |\log t|) < \frac{2}{e}t < t$ for all $t \in (0,1/e)$, and applying this with $t = \frac{x}{2Cn} < \frac{1}{e}$ yields that 
$\left( \frac{x}{2Cn} \right)^2 \left| \log\left( \frac{x}{2Cn} \right)^2 \right| < \frac{x}{2Cn} = \lambda |\log \lambda|$. The monotonicity of $\lambda |\log \lambda|$ on $(0,1/e)$ then implies that $1/e > \lambda > \left( \frac{x}{2Cn} \right)^2 $ which in turn implies the claim in \eqref{loglub}. 
Combining \eqref{Snltbound1} and \eqref{loglub} we get 
\[
P\left( \sum_{i=1}^n \xi_i \leq - x \right) 
\leq  \exp\left\{ \frac{-x^2}{4Cn \log\left( \frac{2Cn}{x} \right)  } \right\} \leq
 \exp\left\{ \frac{-x^2}{4Cn \log n } \right\}, \quad \text{for all } 2C \leq x \leq \frac{2C}{e}n, 
\]
and then by choosing $C'$ large enough we get $P\left( \sum_{i=1}^n \xi_i \leq - x \right) \leq C' \exp\left\{ \frac{-x^2}{8Cn \log n } \right\}$ for all $x \leq \frac{2C}{e} n$. 
Finally, we can extend this bound to $x\leq an$ by changing the coefficient in the exponent. Indeed, if $\frac{2C}{e}n < x \leq an$ then 
\begin{align*}
P\left( \sum_{i=1}^n \xi_i \leq - x \right) 
&\leq P\left( \sum_{i=1}^n \xi_i \leq - \frac{2C}{e} n \right) \\  
&\leq C' \exp\left\{  \frac{-\left(\frac{2Cn}{e}\right)^2}{4Cn \log n } \right\}
\leq C' \exp\left\{ -\frac{C}{e^2a^2} \frac{x^2}{n\log n} \right\}. 
\end{align*}
This completes the proof of the corollary. 
\end{proof}

\begin{lemma}\label{Snltail}
Assume that $\{\xi_i\}_{i\geq 1}$ are i.i.d., zero mean random variables such that $P(|\xi_i| > x) = \mathcal{O}(x^{-2})$, and let $S_n = \sum_{i=1}^n \xi_i$.  
Then, there exists a constant $C>0$ such that 
\[
P(|S_n| > t \sqrt{n \log n}) \leq \frac{C}{t^2 \log n} + \frac{C}{t^4} \qquad \forall t \leq \sqrt{\frac{n}{\log n}}. 
\]
\end{lemma}

\begin{proof}
First of all, note that by choosing $C>0$ large enough we can assume that $t> 1/\sqrt{\log n}$. 
Next, note that 
\begin{align*}
P(|S_n| > t\sqrt{n\log n}) 
&\leq n P\left( |\xi_1| > t \sqrt{n\log n} \right) + P\left( \left| \sum_{i=1}^n \xi_i \ind_{|\xi_i| \leq t \sqrt{n\log n}} \right| > t \sqrt{n\log n} \right). 
\end{align*}
Since the assumption on the tail of $|\xi_1|$ implies that the first term on the right is $\mathcal{O}\left( \frac{1}{t^2 \log n} \right)$, it remains only to bound the second probability for $1/\sqrt{\log n} \leq t\leq \sqrt{n/\log n}$. To this end, first note that 
\[
\left| E\left[ \xi_i \ind_{|\xi_i| \leq x} \right] \right| = 
\left| E\left[ \xi_i \ind_{|\xi_i| > x} \right] \right| 
\leq  E\left[ |\xi_i| \ind_{|\xi_i| > x} \right] = \mathcal{O}( x^{-1} ), 
\]
where the first equality is because $E[\xi_i]=0$ and the last equality follows from the assumptions on the tails of $\xi_i$. Therefore, there exists a constant $L>0$ such that 
\[
\left| E\left[ \sum_{i=1}^n \xi_i \ind_{|\xi_i| \leq t \sqrt{n\log n}} \right] \right|
\leq n \left| E\left[ \xi_i \ind_{|\xi_i| \leq t \sqrt{n\log n}} \right] \right| 
\leq \frac{t}{2}\sqrt{n\log n}
\qquad \text{for all } t > L/\sqrt{\log n}. 
\]
Letting $\bar{\xi}_{i,t,n} = \xi_i \ind_{|\xi_i|\leq t \sqrt{n\log n}} - E\left[\xi_i \ind_{|\xi_i|\leq t \sqrt{n\log n}}\right]$, 
we can conclude that for $t>L/\sqrt{\log n}$ we have 
\begin{align*}
& P\left( \left| \sum_{i=1}^n \xi_i \ind_{|\xi_i| \leq t \sqrt{n\log n}} \right| > t \sqrt{n\log n} \right) 
\leq P\left( \left| \sum_{i=1}^n \bar{\xi}_{i,t,n} \right| > \frac{t\sqrt{n\log n}}{2} \right) \\
&\leq \frac{4}{t^4 n^2 (\log n)^2} E\left[ \left( \sum_{i=1}^n \bar{\xi}_{i,t,n} \right)^4 \right] \\
&\leq \frac{C}{t^4 n^2(\log n)^2} \left\{ n E\left[ \xi_i^4 \ind_{|\xi_i| \leq t\sqrt{n\log n}} \right] + n^2 E\left[ \xi_i^2 \ind_{|\xi_i| \leq t\sqrt{n\log n}} \right]^2  \right\}.
\end{align*}
Since the tail decay assumptions imply that 
$ E\left[ \xi_i^4 \ind_{|\xi_i| \leq x} \right]  = \mathcal{O}( x^2 )$
and 
$E\left[ \xi_i^2 \ind_{|\xi_i| \leq x} \right] = \mathcal{O}( \log x ) $, we can conclude that for $n$ large enough and $t \in [1/\sqrt{\log n}, \sqrt{n/\log n}]$ we have 
\begin{align*}
  P\left( \left| \sum_{i=1}^n \xi_i \ind_{|\xi_i| \leq t \sqrt{n\log n}} \right| > t \sqrt{n\log n} \right) 
  &\leq  \frac{C}{t^2\log n} + \frac{C}{t^4 (\log n)^2} E\left[ \xi_i^2 \ind_{|\xi_i| \leq n} \right]^2 \\
   &\leq  \frac{C}{t^2\log n} + \frac{C}{t^4}.
\end{align*}
\end{proof}

\section{Precise truncated variance asymptotics}\label{app:tvar}

\begin{proof}[Proof of Corollary \ref{cor-tvarlim}]
 First of all, we claim that it is enough to prove the claimed asymptotics for the truncated second moment instead of the truncated variance. That is, we claim it is enough to prove 
 \begin{equation}\label{tsm-prec}
 \lim_{n\to\infty} \sup_{x\geq \sqrt{n}\log^4 n} \left| \frac{\E\left[ (\tilde{Z}_n)^2 \ind_{|\tilde{Z}_n| \leq x}  \right] }{b^2 n\log n + 2K_0 v n \log\left( \frac{x \wedge (nv/2)}{\sqrt{n}} \right) }  - 1 \right| = 0. 
 \end{equation}
To see that it is enough to prove \eqref{tsm-prec}, note that Lemma \ref{cor-L1tail} implies for $n$ sufficiently large and $x\geq \sqrt{n}\log^4 n$ that
\begin{align*}
 \left| \Var(\tilde{Z}_n \ind_{|\tilde{Z}_n| \leq x} ) - \E\left[ (\tilde{Z}_n)^2 \ind_{|\tilde{Z}_n| \leq x}  \right] \right|
 = \E\left[ \tilde{Z}_n \ind_{|\tilde{Z}_n| \leq x}  \right]^2
 = \E\left[ \tilde{Z}_n \ind_{|\tilde{Z}_n| > x}  \right]^2 
 \leq \frac{C n^2}{x^2}. 
\end{align*}
Using this, the statement of the lemma follows easily from \eqref{tsm-prec}. 

The advantage of \eqref{tsm-prec} rather than the original statement in the Lemma is that the truncated second moment $\E\left[ (\tilde{Z}_n)^2 \ind_{|\tilde{Z}_n| \leq x}  \right]$ is monotone in $x$ whereas the truncated variance is not. 
In particular, since $\E\left[ (\tilde{Z}_n)^2 \ind_{|\tilde{Z}_n| \leq x}  \right] \leq \Var(Z_n) \sim (b^2 + K_0 v)n\log n$ then we need only to get good upper bounds on the truncated second moment when $x \in [\sqrt{n}\log^4 n, nv/2]$. 
For such $x$ we then have from Lemma \ref{cor-L2trunc} for $M$ fixed and $n$ sufficiently large that
\begin{align}
\E\left[ (\tilde{Z}_n)^2 \ind_{|\tilde{Z}_n| \leq x}  \right]
&\leq \E\left[ (\tilde{Z}_n)^2 \ind_{|\tilde{Z}_n| \leq M\sqrt{n\log n}}  \right] + \E\left[ (\tilde{Z}_n)^2 \ind_{M\sqrt{n\log n} < |\tilde{Z}_n| < \sqrt{n}\log^4 n }  \right] \label{tsm-ub1}  \\
&\qquad + \E\left[ (\tilde{Z}_n)^2 \ind_{\sqrt{n}\log^4 n \leq |\tilde{Z}_n| \leq x }  \right] \label{tsm-ub2}.
\end{align}
By Lemma \ref{cor-L2trunc} we can bound the second expectation in the last line by $C n \log\log n + \frac{Cn\log n}{M^2}$, while the Bounded Convergence Theorem implies the first expectation is asymptotic to $ b^2 E[\Phi^2\ind_{|\Phi|\leq M}] n\log n$ as $n\to\infty$. Thus, by first choosing $M$ large and then $n$ large enough we get that for any $\e>0$ there exists an $n_0 = n_0(\e)$ such that the two terms on the right in \eqref{tsm-ub1} can be bounded above by $(1+2\e)b^2 n\log n$ for all $n\geq n_0$. For the last term in \eqref{tsm-ub2}, it follows from \eqref{rtail-cen} and \eqref{cprecise ltail} that for any $\e>0$ there is an $n_1 = n_1(\e)$ such that $\P(|\tilde{Z}_n| > t) \leq \frac{(1+\e)K_0 v n}{t^2}$ for all $t \in [\sqrt{n}\log^4 n, nv/2]$. 
Applying this bound we get for all $n$ large enough (depending on $\e$) that
\begin{align*}
\E\left[ (\tilde{Z}_n)^2 \ind_{\sqrt{n}\log^4 n \leq |\tilde{Z}_n| \leq x }  \right]
&\leq n \log^8 n \P\left( |\tilde{Z}_n| > \sqrt{n}\log^4 n \right) + \int_{\sqrt{n}\log^4 n}^{x} 2t \P(|\tilde{Z}_n| > t) \, \ud t \\
&\leq (1+\e)K_0 v n + 2(1+\e) K_0 v n \log\left( \frac{x}{\sqrt{n}\log^4 n} \right)  \\
&\leq (1+2\e)2 K_0 v n \log\left( \frac{x}{\sqrt{n}} \right). 
\end{align*}
Since $\e>0$ is arbitrary this completes the proof of the needed upper bound for $\E[(\tilde{Z}_n)^2 \ind_{|\tilde{Z}_n|\leq x} ]$.

For the lower bound on $\E[(\tilde{Z}_n)^2 \ind_{|\tilde{Z}_n|\leq x} ]$ we note that this truncated second moment is non-decreasing in $x$ and so it's enough to only give the necessary lower bounds for $x \in [\sqrt{n}\log^4 n, vn/2]$.
Using Lemma \ref{cor-smalltsm} we can bound the truncated second moment below by $\E[(\tilde{Z}_n)^2 \ind_{|\tilde{Z}_n|\leq \sqrt{n}\log^4 n} ] \geq (1-2\e)b^2 n\log n$ for any $\epsilon>0$ and $n$ sufficiently large.
This is a good enough lower bound for $x \in [\sqrt{n}\log^4 n, \sqrt{n}e^{\sqrt{\log n}} ]$, but it
remains to get a good lower bound for $x \in [\sqrt{n}e^{\sqrt{\log n}}, vn/2]$. 
For such $x$ we can begin
by noting that Lemma \ref{cor-smalltsm} implies that for any $\e>0$ and $n$ sufficiently large we have
\[
\E\left[ (\tilde{Z}_n)^2 \ind_{|\tilde{Z}_n| \leq x}  \right]
\geq (1-2\e)b^2 n\log n + \E\left[ (\tilde{Z}_n)^2 \ind_{\sqrt{n}\log^4 n \leq |\tilde{Z}_n| \leq x }  \right].
\]
For the second term on the right, it follows from \eqref{cprecise ltail} that $\P(|\tilde{Z}_n| \geq t) \geq \P(\tilde{Z}_n \leq -t) \geq (1-\e)K_0 (nv-t)t^{-2}$ for all $t \in [\sqrt{n}\log^4 n, nv/2]$ and $n$ sufficiently large. 
Therefore, for $n$ sufficiently large we have 
\begin{align*}
 \E\left[ (\tilde{Z}_n)^2 \ind_{\sqrt{n}\log^4 n \leq |\tilde{Z}_n| \leq x }  \right]
 &\geq \int_{\sqrt{n}\log^4 n}^x 2t \P(|\tilde{Z}_n| > t ) \, \ud t - x^2 \P(|\tilde{Z}_n| > x) \\
 &\geq \int_{\sqrt{n}\log^4 n}^x 2 (1-\e)K_0 (nv-t)t^{-1} \, \ud t - (1+\e)K_0 nv \\
 &\geq 2(1-2\e)K_0 v n \log\left( \frac{x}{\sqrt{n}} \right), 
\end{align*}
where in the last inequality we used that $x \geq \sqrt{n}e^{\sqrt{\log n}}$. 
\end{proof}

\providecommand{\bysame}{\leavevmode\hbox to3em{\hrulefill}\thinspace}
\providecommand{\MR}{\relax\ifhmode\unskip\space\fi MR }
\providecommand{\MRhref}[2]{%
  \href{http://www.ams.org/mathscinet-getitem?mr=#1}{#2}
}
\providecommand{\href}[2]{#2}





\begin{thebibliography}{10}

\bibitem{AP16}
Sung~Won Ahn and Jonathon Peterson, \emph{Oscillations of quenched slowdown
  asymptotics for ballistic one-dimensional random walk in a random
  environment}, Electron. J. Probab. \textbf{21} (2016), Paper No. 16, 27.
  \MR{3485358}

\bibitem{Ash00}
R.B. Ash and C.A. Doleans-Dade, \emph{Probability and measure theory}, Elsevier
  Science, 2000.
 
\bibitem{ABF18}
L. Avena, O. Blondel, and A. Faggionato,
\emph{Analysis of random walks in dynamic random environments via {$L^2$}-perturbations},
Stochastic Process. Appl.
\textbf{128} (2018), no. 10, 3490--3530.  \MR{3849817}

  
\bibitem{ACdCdH20}
Luca Avena, Yuki Chino, Conrado da~Costa, and Frank den Hollander.
\newblock Random walk in cooling random environment: recurrence versus
  transience and mixed fluctuations.
\newblock {\em Ann. Inst. Henri Poincar\'e{} Probab. Stat.}, 58(2):967--1009,
  2022.  

\bibitem{ACP21}
Luca {Avena}, Conrado {da Costa}, and Jonathon {Peterson}, \emph{{Gaussian,
  stable, tempered stable and mixed limit laws for random walks in cooling
  random environments}}, arXiv e-prints (2021), arXiv:2108.08396.
  \newblock To appear in {\em Ann. Inst. Henri Poincar\'e{} Probab. Stat.}
  
 \bibitem{AveHol19}
Luca Avena and Frank~den Hollander, \emph{Random walks in cooling random
  environments}, Sojourns in Probability Theory and Statistical Physics - III
  (Singapore) (Vladas Sidoravicius, ed.), Springer Singapore, 2019, pp.~23--42.

\bibitem{BZ06}
Antar Bandyopadhyay and Ofer Zeitouni,
\emph{Random walk in dynamic {M}arkovian random environment},
ALEA Lat. Am. J. Probab. Math. Stat. \textbf{1} (2006), pp.~205--224. \MR{2249655}

\bibitem{BMP00}
  Carlo Boldrighini, Robert Minlos, and Alessandro Pellegrinotti,
\emph{Almost-sure central limit theorem for a Markov model of random
  walk in dynamical random environment.}  Probab Theory Relat Fields
\textbf{109}, (1997), 245–273. 
  
\bibitem{BD18}
Dariusz Buraczewski and Piotr Dyszewski, \emph{Precise large deviations for
  random walk in random environment}, Electron. J. Probab. \textbf{23} (2018),
  Paper no. 114, 26. \MR{3885547}

  

\bibitem{CGZ00}
Francis Comets, Nina Gantert, and Ofer Zeitouni, \emph{Quenched, annealed and
  functional large deviations for one-dimensional random walk in random
  environment}, Probability Theory and Related Fields \textbf{118} (2000),
  no.~1, 65--114.

\bibitem{DPZ96}
Amir Dembo, Yuval Peres, and Ofer Zeitouni, \emph{Tail estimates for
  one-dimensional random walk in random environment}, Comm. Math. Phys.
  \textbf{181} (1996), no.~3, 667--683. \MR{1414305}

\bibitem{Durret}
Rick Durrett, \emph{Probability---theory and examples}, Cambridge Series in
  Statistical and Probabilistic Mathematics, vol.~49, Cambridge University
  Press, Cambridge, 2019, Fifth edition of [ MR1068527]. \MR{3930614}

\bibitem{Fel71}
William Feller, \emph{An introduction to probability theory and its
  applications. {V}ol. {II}.}, Second edition, John Wiley \& Sons Inc., New
  York, 1971.

\bibitem{FGP10}
Alexander Fribergh, Nina Gantert, and Serguei Popov, \emph{On slowdown and
  speedup of transient random walks in random environment}, Probab. Theory
  Related Fields \textbf{147} (2010), no.~1-2, 43--88. \MR{2594347}

\bibitem{GZ98}
Nina Gantert and Ofer Zeitouni, \emph{Quenched sub-exponential tail estimates
  for one-dimensional random walk in random environment}, Comm. Math. Phys.
  \textbf{194} (1998), no.~1, 177--190. \MR{1628294}

\bibitem{HdHdSST15}
Marcelo Hil\'ario, Frank den Hollander, Renato Soares dos Santos,
             Vladas Sidoravicius, and Augusto Teixeira,
\emph{Random walk on random walks},
Electronic Journal of Probability,
\textbf{20},(2015),no.~95, 1--35.  
  
\bibitem{HKT20}
Marcelo R. Hil\'ario, Daniel Kious, and Augusto Teixeira,
\emph{Random walk on the simple symmetric exclusion process},
Comm. Math. Phys.
\textbf{379} (2020), no.~1, 61--101. \MR{4152267}

\bibitem{HS15}
Fran\c cois Huveneers and Fran\c cois Simenhaus, 
\emph{Random walk driven by the simple exclusion process},
Electron. J. Probab.
\textbf{20} (2015), Paper no. 105, 42. \MR{3407222}

\bibitem{dHol00}
Frank den Hollander, \emph{Large deviations}, FIM Series, American Mathematical
  Society, 2000.

\bibitem{KKS75}
H.~Kesten, M.~V. Kozlov, and F.~Spitzer, \emph{{A limit law for random walk in
  a random environment}}, Compositio Math. \textbf{30} (1975), 145--168.

\bibitem{Petrov}
V.~V. Petrov, \emph{Sums of independent random variables}, Springer-Verlag, New
  York-Heidelberg, 1975, Translated from the Russian by A. A. Brown, Ergebnisse
  der Mathematik und ihrer Grenzgebiete, Band 82. \MR{0388499}

\bibitem{Sinai}
Ya.~G. Sina\u{\i}, \emph{The limit behavior of a one-dimensional random walk in
  a random environment}, Teor. Veroyatnost. i Primenen. \textbf{27} (1982),
  no.~2, 247--258. \MR{657919}

\bibitem{S75}
Fred Solomon, \emph{Random walks in a random environment}, Ann. Probab.
  \textbf{3} (1975), no.~1, 1--31.

\bibitem{SZ99}
Alain-Sol Sznitman and Martin Zerner, \emph{{A law of large numbers for random
  walks in random environment}}, Ann. Probab. \textbf{27} (1999), no.~4,
  1851--1869. \MR{1742891}

\bibitem{Yon19}
Yongjia {Xie}, \emph{{Functional Weak Limit of Random Walks in Cooling Random
  Environment}}, Electron. Commun. Probab. \textbf{25} (2020), 1--14.

\bibitem{Z04}
O.~Zeitouni, \emph{Random {W}alks in {R}andom {E}nvironment}, Lecture Notes in
Mathematics \textbf{1837} (2004), 1--121.




  

\end{thebibliography}
\end{document}


